\documentclass[12pt]{article}
\usepackage{amsfonts,amsmath,amsthm, amssymb}
\usepackage{latexsym, euscript, epic, eepic}
\usepackage{color}
\newtheorem{theorem}{Theorem}[section]

\newtheorem{defn}{Definition}[section]
\newtheorem{prop}[theorem]{Proposition}
\newtheorem{rem}[theorem]{Remark}
\newtheorem{cro}{Corollary}[section]
\newtheorem{thm}{Theorem}[section]
\newtheorem{lem}{Lemma}[section]

\begin{document}

\title{Multifractal analysis for historic set in topological
dynamical systems
%\footnotetext {* Corresponding author}
 }
\author{Xiaoyao Zhou $^\dag,$  Ercai Chen$^{\dag \ddag}$\\
\small \it $\dag$ School of Mathematical Sciences and Institute of Mathematics, Nanjing Normal University,\\
\small \it Nanjing 210023, P.R.China,
\small \it $\ddag$ Center of Nonlinear Science,\\
\small \it Nanjing University, Nanjing 210093,  P.R.China.\\
\small \it e-mail:
\small \it\noindent zhouxiaoyaodeyouxian@126.com\\
\small \it \noindent ecchen@njnu.edu.cn}
\date{}
\maketitle

\noindent\textbf{Abstract.} In this article,  the historic set is
divided into different level sets and we use topological pressure to
describe the size of these level sets. We give an application of
these results to dimension theory. Especially,  we use topological
pressure to describe the relative multifractal spectrum of ergodic
averages and give a positive answer to the conjecture posed by L.
Olsen (J. Math. Pures Appl. {\bf 82} (2003)).


\section{Introduction}

A topological dynamical system is a triple $(X,d,T)$ (or tuple
$(X,T)$ for short) consisting of a compact metric space $(X,d)$  and
a continuous map $T:X\to X.$

An orbit $\{x,T(x),T^2(x),\cdots\}$ has historic behavior if for
some continuous function $\psi:X\to\mathbb{R},$ the average

\begin{equation*}
\lim\limits_{n\to\infty}\frac{1}{n}\sum\limits_{i=0}^{n-1}\psi(T^i(x))
\end{equation*}
does not exist. This terminology was introduced by Ruelle in
\cite{Rue}. If this limit does not exist, it follows that 'partial
averages' $\lim_{n\to\infty}\frac{1}{n}\sum_{i=0}^{n-1}\psi(T^ix)$
keep change considerably so that their values give information about
the epoch to which $n$ belongs.

The problem, whether there are persistent classes of smooth
dynamical systems such that the set of initial states which give
rise to orbits with historic behavior has 'positive Lebesgue
measure' was discussed by Ruelle  \cite{Rue}. Takens also
investigated the problem in the survey  \cite{Tak}.

 Very recently,
the idea of multifractal analysis plays an important role in the
study of dynamical system. V. Climenhaga \cite{Cli} considered the
topological pressure  function on the level sets of asymptotically
defined quantities in a topological dynamical systems. D. Feng and
W. Huang \cite{FenHua} studied the general asymptotically
sub-additive on general topological dynamical systems and
established some variational relations between the topological
entropy of the level sets of Lyapunov exponents, measure-theoretic
entropies and topological pressures in this general situation.

In this article, we will use the framework introduced by Olsen to
investigate the geometric structure of the historic set in view of
multifractal analysis.

Denote by $M(X), M(X,T)$ and $E(X,T),$ the set of all Borel
probability measures on $X,$
 the collection of all $T$-invariant Borel probability measures,
and  the set of all ergodic $T$-invariant Borel probability
measures, respectively.

 It is well-known that $M(X)$ and
$M(X,T)$ are both  convex, compact spaces endowed with weak*
topology. For $\mu, \nu\in M(X),$  define a compatible metric $\rho$
on $M(X)$ as follows:

\begin{equation*}
\rho(\mu,\nu):=\sum\limits_{k\geq1}\frac{|\int_X f_{k}d\mu-\int_X
f_kd\nu|}{2^{k}},
\end{equation*}
where $\{f_1,f_2,\cdots\}$ is a countable and dense in $C(X,[0,1]).$
Note that $\rho(\mu,\nu)\leq 1,$ for any $\mu,\nu\in M(X).$ This
article uses an equivalent metric on $X,$ still denoted by $d$,

\begin{equation*}
d(x,y):=\rho(\delta_{x},\delta_{y})
\end{equation*}
for convenience. For $n\in\mathbb{N}$, let $L_{n}:X\to M(X)$ be the
$n$-th empirical measure, i.e.,

\begin{equation*}
L_{n}x=\frac{1}{n}\sum\limits_{k=0}^{n-1}\delta _{T^{k}x},
\end{equation*}
where $\delta_{x}$ denotes the Dirac measure at $x.$ Let $\Xi$ be a
continuous affine map from $M(X)$ to a vector space $Y$ with a
linear compatible metric $d'$. $(Y,\Xi)$ is called a deformation of
$L_{n}.$ Let  $A(x_{n})$ be the set of accumulation points of
$\{x_{n}\}$ and $
 D(T,\Xi) $ be the set consists of  the points $x$ such that $ \lim\limits_{n}\Xi L_{n}x$ does not
exist. $D(T,\Xi)$ is called the historic set  for $(X,T).$

This article is devoted to investigate the structure of $ D(T,\Xi)$
via the following framework introduced and developed by Olsen
\cite{Ols1}, \cite{Ols2}, \cite{Ols3}, \cite{Ols4}
 and Olsen \& Winter \cite{OlsWin}.

 More precisely,
for a subset $C$ of $Y,$ this article uses topological pressure to
describe the size of the following so-called sup set, equ set and
sub set:

\begin{eqnarray*}\begin{split}
&\Delta_{sup}(C)=\{x\in X:A(\Xi L_{n}x)\subset
C\},\\&\Delta_{equ}(C)=\{x\in X:A(\Xi L_{n}x)=
C\},\\&\Delta_{sub}(C)=\{x\in X:A(\Xi L_{n}x)\supset C\}.
\end{split}\end{eqnarray*}
Such sets together give us a complete description of the dynamics of
the historic set and provides the basis for a substantially better
understanding of the underlying geometry of the historic set. More
generally, for $S_{1},S_{2}\subset Y,$ considering
$\Delta(S_{1},S_{2})=\{x\in X:S_{1}\subset A(\Xi L_{n}x)\subset
S_{2}\},$ we have
\begin{eqnarray*}\begin{split}
&\Delta(\emptyset,C)=\Delta_{sup}(C);\\&\Delta(C,C)=\Delta_{equ}(C);\\&\Delta(C,Y)=\Delta_{sub}(C).
\end{split}\end{eqnarray*}
Obviously, multifractal analysis is a special case of this
framework. For example, for any $\phi\in C(X,\mathbb{R}),$ choose
$Y=\mathbb{R},$ and define $\Xi: M(X)\to \mathbb{R}$ by $
\Xi:\mu\mapsto\int\phi d\mu.$ Then for $
C=\{\alpha\}\subseteq\mathbb{R},$ it follows that

\begin{equation}\label{formula1}
\Delta_{equ}(C)=\left\{x\in
X:\lim\limits_{n\to\infty}\frac{1}{n}\sum\limits_{k=0}^{n-1}f(T^kx)=\alpha\right\},
\end{equation}
and

\begin{equation}\label{formula2}
 D(T,\Xi)=\left\{x\in X:{\rm the~ limit~} \lim\limits_{n\to\infty}\frac{1}{n}\sum\limits_{k=0}^{n-1}f(T^kx) {\rm~does ~not~
exist}\right\}.
\end{equation}
 Previous studies \cite{BarSau},
\cite{BarSauSch}, \cite{FanFen}, \cite{PeiChe1}, \cite{PfiSul2},
\cite{TakVer}, \cite{Tho1}, \cite{Yam} have obtained a number of
fruitful results regarding different quantities to describe
 the size of (\ref{formula1}) in some dynamical systems with some mixing
properties. The quantities include  Hausdorff dimension, packing
entropy,  topological entropy and topological pressure. Dynamical
systems can be symbolic spaces or satisfy some mild conditions such
as the specification property or  the g-almost product property and
so on.

At the beginning, $D(T,\Xi)$ had been considered of little interest
in dynamical systems and geometric measure theory
 due to the fact that $\mu( D(T,\Xi))=0$ for any
$\mu\in M(X,T).$ However, recent work  \cite{BarSch},
\cite{CheTasShu}, \cite{PeiChe2}, \cite{Tho2}, \cite{ZhoChe} has
changed such attitudes. Hausdorff dimension or topological entropy
or topological pressure of (\ref{formula2}) can be large enough even
equal to that of the whole space.
 It illustrates that the historic set has rich
information. Hence, it is meaningful to divide the historic set into
different level sets and investigate these level sets. A series
results in symbolic space and iterated function system can be found
in \cite{BaeOlsSni}, \cite{Ols1}, \cite{Ols2}, \cite{Ols3},
\cite{Ols4}. This article divide $D(T,\Xi)$ into different level
sets $\Delta_{equ}(\cdot)$ and $\Delta_{sub}(\cdot).$

This investigation uses topological pressure to describe
$\Delta_{equ}(\cdot), \Delta_{sub}(\cdot)$ and so on. Topological
pressure is a powerful tool and is not only a generalization of
topological entropy but also closely related to Hausdorff dimension.
This article discusses the dynamical systems satisfying g-almost
product property and the uniform separation property that were
introduced by C. Pfister and W. Sullivan \cite{PfiSul2}. These two
properties are strictly weaker than the specification property and
the positive expansive property. (For example, all $\beta$-shifts
have the g-almost product property and the uniform separation
property is true for expansive and more generally asymptotically
h-expansive maps.)

As an application of our results, we study symbolic spaces and
iterated function systems. We stress that the metric in symbolic
space here is ultrametric rather than the metric in \cite{Ols2}.

\noindent For $\varphi\in C(X,\mathbb{R}),$ define
\begin{equation} \Lambda(y,\varphi)=\left\{
\begin{array}{lr}
\sup\limits_{\mu\in M(X,T)\atop\Xi \mu=y}\{h(T,\mu)+\int\varphi d\mu\},  {~ \rm for~} &y\in\Xi(M(X,T))  \\
-\infty,  &{\rm otherwise}.
\end{array}\right.
\end{equation}
We state our main theorems as below:

\begin{thm}\label{thm1.1}
$(X,T,\Xi,L_{n},Y)$ satisfies g-almost product property and the
uniform separation property and $\varphi\in C(X,\mathbb{R}).$ If
\begin{enumerate}
  \item $C\subset Y$
is not a compact and connected subset of  $\Xi(M(X,T)),$ then

\begin{equation*}
\{x\in X:A(\Xi L_{n}x)=C \}=\emptyset,
\end{equation*}

  \item $C\subset Y$ is a compact and connected subset of $\Xi(M(X,T)),$ then

\begin{equation*}
P(\Delta_{equ}(C),\varphi)=\inf\limits_{y\in C}\sup\limits_{\mu\in
M(X,T)\atop\Xi \mu=y}\left\{h(T,\mu)+\int\varphi
d\mu\right\}=\inf\limits_{y\in C}\Lambda(y,\varphi).
\end{equation*}
\end{enumerate}
\end{thm}

\begin{thm}\label{thm1.2}
 $(X,T,\Xi,L_{n},Y)$ as before and $\varphi\in C(X,\mathbb{R}).$
If
\begin{enumerate}
  \item  $C\subset
Y$ is not a subset of $\Xi(M(X,T)),$ then

\begin{equation*}
\{x\in X:C\subset A(\Xi L_{n}x)\}=\emptyset,
\end{equation*}
  \item   $C\subset Y$ is  a subset of $\Xi(M(X,T)),$ then

\begin{equation*}
P(\Delta_{sub}(C),\varphi)=\inf\limits_{y\in C}\sup\limits_{\mu\in
M(X,T)\atop\Xi\mu=y}\left\{h(T,\mu)+\int\varphi
d\mu\right\}=\inf\limits_{y\in C} \Lambda(y,\varphi).
\end{equation*}
\end{enumerate}
\end{thm}

\begin{thm}\label{thm1.3}
$(X,T,\Xi,L_{n},Y)$ as before and $\varphi\in C(X,\mathbb{R}),$ fix
$S_{1}\subset\Xi(M(X,T)),S_{2}\subset Y,$ if
\begin{enumerate}
  \item $S_{1}=\emptyset$, then

\begin{equation*}
P(\Delta(S_{1},S_{2}),\varphi)=\sup\limits_{x\in
S_{2}}\Lambda(x,\varphi),
\end{equation*}
  \item   $S_{1}\neq\emptyset$ and $S_{1}$ is contained in a connected
component of $S_{2},$ then

\begin{equation*}
\sup\limits_{S_{1}\subset Q\subset
S_{2}\atop~Q\subseteq\Xi(M(X,T))~is~ compact~and~
connected}\inf\limits_{x\in Q}\Lambda(x,\varphi)\leqslant
P(\Delta(S_{1},S_{2}),\varphi)\leqslant\inf\limits_{x\in
S_{1}}\Lambda(x,\varphi),
\end{equation*}

  \item  $S_{1}\neq\emptyset$ and $S_{1}$ is not contained in a
connected component of $S_{2},$ then

\begin{equation*}
\{x\in X:S_{1}\subset A(\Xi L_{n}x)\subset S_{2}\}=\emptyset .
\end{equation*}
\end{enumerate}
\end{thm}

\begin{thm}\label{thm1.4}
 $(X,T,\Xi,L_{n},Y)$ as before and
$\varphi\in C(X,\mathbb{R}),$ fix
$S_{1}\subset\Xi(M(X,T)),S_{2}\subset Y,$
\begin{enumerate}
  \item  If $S_{1}=\emptyset,$
then

\begin{equation*}
P(\Delta(S_{1},S_{2}),\varphi)=\sup\limits_{x\in
S_{2}}\Lambda(x,\varphi),
\end{equation*}

  \item If $S_{1}\neq\emptyset$ and $\overline{co}(S_{1})$ the closed convex hull
of $S_{1}$ is contained in a connected component of $S_{2},$ then

\begin{equation*}
P(\Delta(S_{1},S_{2}),\varphi)=\inf\limits_{x\in
S_{1}}\Lambda(x,\varphi),
\end{equation*}

  \item If $S_{1}\neq\emptyset$ and $S_{1}$ is not contained in a connected
component of $S_{2},$ then

\begin{equation*}
\{x\in X:S_{1}\subset A(\Xi L_{n}x)\subset S_{2}\}=\emptyset.
\end{equation*}
\end{enumerate}
\end{thm}

\section{Preliminaries}

A remark about notations is presented here for convenience.
\begin{rem}{\rm Let $(X,T)$ be a topological dynamical system.
\begin{itemize}
  \item Let $F\subset M(X)$ be a neighborhood, set
$X_{n,F}:=\{x\in X: L_{n}x\in F\}.$

\item Given $\delta>0$ and $\epsilon>0,$ two points $x$ and $y$
are $(\delta,n,\epsilon)$-separated if
$\#\{j:d(T^{j}x,T^{j}y)>\epsilon,0\leq j\leq n-1\}\geqslant \delta
n.$ A subset $E$ is $(\delta,n,\epsilon)$-separated if any pair of
different points of $E$ are $(\delta,n,\epsilon)$-separated.

\item Let $F\subset M(X)$ be a neighborhood of $\nu,$ and
$\epsilon>0,$ set

$N(F;n,\epsilon):=${\rm maximal cardinality of an}
 $(n,\epsilon)$-{\rm separated  subset of } $X_{n,F};$

$N(F;\delta,n,\epsilon):=${\rm maximal cardinality of
 an }$(\delta,n,\epsilon)$-{\rm separated subset of }$X_{n,F}.$

 \item Given $x\in X,$ set $B_{n}(x,\epsilon):=\{y\in X:
d_n(x,y)\leq\epsilon\},$ where
$d_n(x,y)=\max\limits_{i=0,\cdots,n-1} d(T^ix,T^iy).$

\item
A point $x\in X,\epsilon$-shadows a sequence
$\{x_{0},x_{1},\cdots,x_{k}\}$ if $d(T^{j}x,x_{j})\leq
\epsilon~\forall j=0,1,\cdots k.$

\item  Let
$g:\mathbb{N}\rightarrow\mathbb{N}$ be a given nondecreasing unbound
map with the properties $g(n)<n$ and
$\lim\limits_{n\rightarrow\infty}\frac{g(n)}{n}=0.$ The function $g$
is called blow-up function. Given $x\in X$ and $\epsilon>0$. The
g-blow-up of $B_{n}(x,\epsilon)$ is the closed set
\begin{equation*}
B_{n}(g;x,\epsilon):=\{y\in X:\exists\Lambda\subset\Lambda_{n},\#
(\Lambda_{n}\setminus \Lambda)\leqslant
g(n){\rm~and~}\max\{d(T^{j}x,T^{j}y):j\in\Lambda\}\leq\epsilon\},
\end{equation*}
where $\Lambda_{n}=\{0,1,\cdots,n-1\}.$

\item  { (i)}Given $K\subset M(X,T),$ set $G_{K}:=\{x\in X:A(L_{n}(x))=K\}.$

 {(ii)} Given $K^{'}\subset
\Xi M(X,T),$ set $G_{K^{'}}^{*}:=\{x\in X:A(\Xi L_{n}(x))=K^{'}\}.$

 { (iii)} Given $K\subset M(X),$ set $^{K}G:=\{x\in X:A(L_{n}(x))\bigcap K\neq \emptyset\}.$

 { (iv)} Given $K^{'}\subset Y,$ set $^{K^{'}}G^{*}:=\{x\in X:A(\Xi L_{n}(x))\bigcap K^{'}\neq
\emptyset\}.$

\end{itemize}
}\end{rem}

\begin{defn} \label{defn2.1} {\rm \cite{PfiSul2}}
The dynamical system $(X,d,T)$ has the g-almost product property
with blow-up function g, if there exists a nonincreasing function
$m:\mathbb{R}^{+}\rightarrow\mathbb{N}$, such that for any
$k\in\mathbb{N}$, any $x_{1}\in X,\cdots,x_{k}\in X$, any positive
$\epsilon_{1},\epsilon_{2}\cdots\epsilon_{k}$ and any integers
$n_{1}\geq m(\epsilon_{1}),\cdots ,n_{k}\geq m(\epsilon_{k}),$

\begin{equation*}
\bigcap\limits_{j=1}^{k}T^{-M_{j-1}}B_{n_{j}}(g;x_{j},\epsilon_{j})\neq\emptyset,
\end{equation*}
where $M_{0}=0,M_{i}=n_{1}+n_{2}+\cdots+n_{i},i=1,2,\cdots,k-1.$
\end{defn}

\begin{defn} \label{defn2.2} {\rm \cite{PfiSul2}}
The dynamical system $(X,d,T)$ has uniform separation property if
for any $\eta,$ there exist $\delta^{*}>0$ and $\epsilon^{*}>0$ so
that for $\mu$ ergodic and any neighborhood $F\subset M(X)$ of
$\mu,$ there exists $n_{F,\mu,\eta}^{*},$ such that for $n\geq
n_{F,\mu,\eta}^{*},$ $N(F;\delta^{*},n,\epsilon^{*})\geq
\exp(n(h(T,\mu)-\eta)),$ where $h(T,\mu)$ is the metric entropy of
$\mu.$
\end{defn}

\begin{prop} \label{prop2.1} {\rm \cite{PfiSul2}}
 Suppose that
$(X,d,T)$has the $g$-almost product property. Let
$x_{1},\ldots,x_{k}\in X$ $\epsilon_{1}>0,\ldots,\epsilon_{k}>0$,
and $n_{1}\geq m(\epsilon_{1}),\ldots,n_{k}\geq m(\epsilon_{k})$ be
given. Assume that $L_{n_{j}}(x_{j})\in B(\nu_{j},\zeta_{j}),0\leq
j\leq k.$ Then for any $y\in
\bigcap_{i=1}^{k}T^{-M_{i-1}}B_{n_{i}}(g;x_{i},\epsilon_{i})$ and
any probability measure $\alpha,$

\begin{equation*}
\rho(L_{M_{k}}(y),\alpha)\leq
\sum_{j=1}^{k}\frac{n_{j}}{M_{k}}(\zeta_{j}^{'}+\rho(\nu_{j},\alpha)),
\end{equation*}
where
$M_{j}=n_{1}+\dots+n_{j},\zeta_{j}^{'}=\zeta_{j}+\epsilon_{j}+\frac{g(n_{j})}{n_{j}},
j=1,\dots,k.$
\end{prop}

\begin{defn}{\rm\cite{PesPit}}
Given  $Z\subset X,\varphi\in C(X,\mathbb{R}),$ and  let
$\Gamma_{n}(Z,\epsilon)$ be the collection of all finite or
countable covers of $Z$ by sets of the form $B_{m}(x,\epsilon),$
with $m\geq n$. Let
$S_{n}\varphi(x):=\sum_{i=0}^{n-1}\varphi(T^{i}x).$ Set

\begin{equation*}
M(Z,t,\varphi,n,\epsilon):=\inf_{\mathcal{C}\in\Gamma_{n}(Z,\epsilon)}\left\{\sum_{B_{m}(x,\epsilon)\in
\mathcal{C}}\exp \left(-tm+\sup_{y\in
B_{m}(x,\epsilon)}S_{m}\varphi(y)\right)\right\},
\end{equation*}
and

\begin{equation*}
M(Z,t,\varphi,\epsilon)=\lim_{n\to\infty}M(Z,t,\varphi,n,\epsilon),
\end{equation*}
then there exists a unique number $P(Z,\varphi,\epsilon)$ such that

\begin{equation*}
P(Z,\varphi,\epsilon)=\inf\{t:M(Z,t,\varphi,\epsilon)=0\}=\sup\{t:M(Z,t,\varphi,\epsilon)=\infty\}.
\end{equation*}
$P(Z,\varphi)=\lim_{\epsilon\to0}P(Z,\varphi,\epsilon)$ is called
the topological pressure of $Z$.
\end{defn}
It is obvious that  the following hold:

\begin{enumerate}
  \item $P(Z_{1},\varphi)\leqslant P(Z_{2},\varphi)$ for any
$Z_{1}\subset Z_{2}\subset X$;
  \item $P(Z,\varphi)=\sup\limits_{i}P(Z_{i},\varphi)$ , where $Z=\bigcup_{i}Z_{i}\subset
X.$
\end{enumerate}

\begin{defn}{\rm\cite{CheTasShu}}
 If $Y$ is a vector space and $d'$ is a metric in $Y,$ then  $d'$ is linearly compatible if

{\rm(1)}For all $x_{1},x_{2},y_{1},y_{2}\in Y,$
$d'(x_{1}+x_{2},y_{1}+y_{2})\leqslant
d'(x_{1},y_{1})+d'(x_{2},y_{2}).$

{\rm(2)}For all $x,y\in Y$ and all $\lambda\in\mathbb{R},$
$d'(\lambda x,\lambda y)\leqslant|\lambda|d'(x,y).$
\end{defn}
In fact, if $d'$ is induced by a norm, then $d'$ is linearly
compatible.

Now, we present several propositions about metrics $d'$ and $\rho.$
\begin{prop}
 Assume that $d'$ is a linearly compatible metric in $Y.$ Let
$$V(\Xi,\epsilon):=\sup\limits_{\mu,\nu\in M(X)\atop\rho(\mu,\nu)<\epsilon}d'(\Xi\mu,\Xi\nu),$$
then $V(\Xi,\epsilon)\to0$  as $\epsilon\to0.$
\end{prop}

\begin{proof}
It is because that $\Xi:M(X)\to Y$ is continuous and $M(X)$ is
compact.
\end{proof}

\begin{prop}
For any $x\in X,$ any $\epsilon>0,$ there exist sufficiently large
  $N,$ such that for all $n>N, $ we have $\rho(L_{n}x,L_{n+1}x)\leq\epsilon.$
 \end{prop}

\begin{proof}
Choose sufficiently large
  $N,$ such that $\frac{1}{N+1}\leq\epsilon.$ Then

\begin{eqnarray*}\begin{split}
 \rho(L_{n}x,L_{n+1}x)&
 =\sum\limits_{k\geq1}2^{-k}|\int f_{k}dL_{n}x-\int f_{k}dL_{n+1}x|\\&
 =\sum\limits_{k\geq1}2^{-k}|\int f_{k}dL_{n}x-\int\frac{n}{n+1} f_{k}dL_{n}x-\int\frac{1}{n+1} f_{k}d\delta_{T^{n}x}|\\&
=\sum\limits_{k\geq1}2^{-k}\frac{1}{n+1}|\int f_{k}dL_{n}x-\int
f_{k}d\delta_{T^{n}x}|\\&
=\frac{1}{n+1}\rho(L_{n}x,\delta_{T^{n}x})\\&\leq\frac{1}{n+1}\leq\frac{1}{N+1}\leq\epsilon.
\end{split}
\end{eqnarray*}
\end{proof}

\begin{prop}
For any $x\in X,$ any $\epsilon>0,$  there exist sufficiently large
  $N,$ such that for all $n>N, $ we have
   $ x'\in B_{n}(g,x,\epsilon)$ implies
  $\rho(L_{n}x,L_{n}x')<2\epsilon.$
  \end{prop}

\begin{proof}
Since $\lim\limits_{n\to\infty}\frac{g(n)}{n}=0,$ we have that for
any $\epsilon>0,$ there exists a sufficiently large $N\in\mathbb{N}$
such that $\frac{g(n)}{n}<\epsilon$ whenever $n>N.$ Then
\begin{eqnarray*}\begin{split}
\rho(L_{n}x,L_{n}x')&=\sum\limits_{k\geq1}2^{-k}|\int
f_{k}dL_{n}x-\int f_kdL_{n}x'|\\&\leq
\sum\limits_{k\geq1}2^{-k}\frac{g(n)}{n}+\frac{n-g(n)}{n}\epsilon\\&\leq2\epsilon.
\end{split}
\end{eqnarray*}
\end{proof}

\begin{prop}{\rm\cite{Win}}\label{prop2.5}
For any $x\in X, A(\Xi L_{n}x)$ is a compact and connected subset of
$Y.$
\end{prop}

\section{Upper and lower bounds for $P(G_{K'}^{*},\varphi)$}
This section is to show the upper and lower bounds for
$P(G_{K'}^{*},\varphi).$

\begin{prop}\label{prop3.1} {\rm\cite{PeiChe1}}
Let $(X,d,T)$ be a dynamical system,

 \noindent{\rm(i)} if
$K\subseteq M(X,T)$ is a closed subset, then

\begin{equation*}
P(^{K}G,\varphi)\leqslant \sup\left\{h(T,\mu)+\int\varphi
d\mu:\mu\in K\right\}.
\end{equation*}

\noindent {\rm(ii)} if $\mu\in M(X,T),$ then

\begin{equation*}
P(G_{\mu},\varphi)\leqslant h(T,\mu)+\int\varphi d\mu.
\end{equation*}

 \noindent{\rm(iii)}
if $K\subseteq M(X,T)$ is a non-empty closed set, then

\begin{equation*}
P(G_{K},\varphi)\leq\inf\left\{h(T,\mu)+\int\varphi d\mu:\mu\in
K\right\}.
\end{equation*}
\end{prop}
\noindent By the above proposition, we get the following theorem.

\begin{thm}\label{thm3.1}
Let $(X,d,T)$ be a dynamical system,

\noindent {\rm(1)} if $K'\subset Y$ is a closed subset, then

\begin{equation*}
P(^{K'}G^{*},\varphi)\leq\sup\{\Lambda(y,\varphi):y\in K'\},
\end{equation*}

\noindent {\rm(2)} if $K'\subset Y$ is a non-empty closed set, then

\begin{equation*}
P(G_{K'}^{*},\varphi)\leq \inf\{\Lambda(y,\varphi):y\in K'\}.
 \end{equation*}
\end{thm}
\begin{proof}
(1) If $K'\subset Y$ is a closed subset, then

\begin{eqnarray*}\begin{split}
^{K'}G^{*}&=\left\{x\in X:A(\Xi (L_{n}x))\bigcap
K'\neq\emptyset\right\}\\&=\left\{ x\in X:\Xi (AL_{n}x)\bigcap
K'\neq\emptyset\right\}\\&=\left\{ x\in X: (AL_{n}x)\bigcap
\Xi^{-1}K'\neq\emptyset\right\}\\&=^{\Xi^{-1} K'\bigcap M(X,T)}G.
\end{split}
\end{eqnarray*}

hence,
\begin{eqnarray*}\begin{split} P(^{K'}G^{*},\varphi)&=P(^{\Xi
^{-1}K'\bigcap M(X,T)}G,\varphi)\\&\leq
\sup\left\{h(T,\mu)+\int\varphi d\mu:\mu\in\Xi^{-1}K'\bigcap
M(X,T)\right\}\\&=\sup\{\Lambda(y,\varphi):y\in K'\}.
\end{split}
\end{eqnarray*}

(2) If $K'\nsubseteq \Xi M(X,T),$ then $G_{K'}^{*}=\emptyset.$ In
this case, there exists $y\in K'\setminus \Xi M(X,T)$ such that
$\Lambda(y,\varphi)=-\infty.$ If  $K'\subseteq \Xi M(X,T),$ then for
any $y\in K',$ we have

\begin{eqnarray*}\begin{split}
G_{K'}^{*}&=\{x\in X:A(\Xi L_{n}x)= K'\}\\&\subseteq\left\{x\in
X:A(\Xi L_{n}x)\bigcap \{y\}\neq\emptyset\right\}\\&= ^{\{y\}}G^{*}.
\end{split}
\end{eqnarray*}
Then,

\begin{equation*}
 P(G_{K'}^{*},\varphi)\leq
P(^{\{y\}}G^{*},\varphi)\leq\Lambda(y,\varphi),
\end{equation*}
for any $y\in K'.$ In summary,
\begin{equation*}
P(G_{K'}^{*},\varphi)\leq \inf\{\Lambda(y,\varphi):y\in K'\}.
\end{equation*}
\end{proof}
To obtain the lower bound of $P(G_{K'}^{*},\varphi),$ we need endow
dynamical system with some mild conditions.

\begin{prop}{\rm\cite{PfiSul2}}\label{prop3.2}
 Assume that $(X,d,T)$ has the
g-almost product property and the uniform separation
 property. For any $\eta,$ there exists $\delta^{*}$ and
 $\epsilon^{*}>0$ such that for $\mu\in M(X,T)$ and any neighborhood
 $F\subset M(X)$ of $\mu,$ there exists $n_{F,\mu,\eta}^{*},$ such
 that

\begin{equation*}
N(F;\delta^{*},n,\epsilon^{*})\geq \exp(n(h(T,\mu)-\eta)),
\end{equation*}
whence $n\geqslant
 n_{F,\mu,\eta}^{*}.$ Furthermore, for any $\mu\in M(X,T),$

\begin{equation*}
h(T,\mu)\leqslant\lim\limits_{\epsilon\rightarrow0}\lim\limits_{\delta\rightarrow0}\inf\limits_{F\ni\mu}\liminf\limits_{n\to\infty}\frac{1}{n}\log
N(F;\delta,n,\epsilon).
\end{equation*}
\end{prop}

\begin{lem}{\rm\cite{PfiSul2}}\label{lem3.1}
If $K'$ is a connected, non-empty and compact subset of $\Xi
(M(X,T)),$ then there exists a
 sequence $\{\alpha''_{1},\alpha''_{2},\cdots\}$ in $K'$ such that

 \begin{equation*}
\overline{\{\alpha''_{j}:j\in\mathbb{N},j>n\}}=K',
 \end{equation*}
 for any
$n\in\mathbb{N},$ and
$\lim\limits_{j\rightarrow\infty}d'(\alpha''_{j},\alpha''_{j+1})=0.$
\end{lem}

\begin{thm}\label{thm3.2}
 Let $(X,d,T)$ be a dynamical system with the uniform separation and g-almost product property, $\varphi\in C(X,\mathbb{R}).$
If $K'$ is a connected, non-empty and compact subset of $\Xi
 (M(X,T)),$ then

\begin{equation*}
 \inf\limits_{y\in K'}\Lambda(y,\varphi)\leq
 P(G_{K'}^{*},\varphi).
 \end{equation*}
\end{thm}

\begin{proof}
 Let $\eta>0$ and $h^{*}:=\inf\limits_{y\in
K'}\Lambda(y,\varphi)-\eta.$ For any $s<h^{*},$ set
$h^{*}-s:=2\delta>0.$ Given sequence $\{\alpha''_{k}\}$ as in Lemma
\ref{lem3.1}, we construct a subset $G$ such that for each $x\in
G,\{\Xi L_{n}x\}$ has the same limit-point set as the sequence
$\{\alpha''_{k}\} ,$ and $P(G,\varphi)\geq h^{*}.$ For
$\frac{\eta}{2}$ and  $\alpha''_{k}\in K',$ there exists $
\alpha_{k}\in\Xi^{-1}\{\alpha''_{k}\}\subseteq M(X,T)$ such that $
\Lambda(\alpha''_{k},\varphi)\leq h(T,\alpha_{k})+\int \varphi
d\alpha_{k} +\frac{\eta}{2}.$ By Proposition \ref{prop3.2}, it is
easy to obtain that for $\frac{\eta}{2}>0,$ there exist
$\delta^{*}>0$ and $\epsilon^{*}>0$ such that for any neighborhood
$F''\subset\Xi(M(X))$ of $\alpha''_{k}$ (choose $
F''=B(\alpha''_{k},\zeta''_{k})$) there exist
$B(\alpha_{k},\zeta_{k})\subseteq\Xi^{-1}F''$ and
$n^{*}_{B(\alpha_{k},\zeta_{k}),\alpha_{k},\frac{\eta}{2}}$
satisfying

\begin{eqnarray}\label{formula4}
N(B(\alpha_{k},\zeta_{k});\delta^{*},n,\epsilon^{*})\geq
\exp(n(h(T,\alpha_{k})-\frac{\eta}{2})),
\end{eqnarray}
whence $n\geq
n_{B(\alpha_{k},\zeta_{k}),\alpha_{k},\frac{\eta}{2}}^{*}$ and
$\zeta_k, \zeta''_k$ will be determined later.

\noindent Choose three strictly decreasing sequences
$\{\zeta_{k}\}_k, \{\zeta''_{k}\}_k$ and $\{\epsilon_{k}\}_k,$ such
that

\noindent{\rm (i)} $\lim\limits_{k}\zeta_{k}=0,$
$\lim\limits_{k}\zeta''_{k}=0$ and $\lim\limits_{k}\epsilon_{k}=0,$

\noindent{\rm (ii)} $\epsilon_{1}<\epsilon^{*}$ and $|\int\varphi
d\alpha_{k}-\int \varphi d\mu|\leq\frac{\delta}{6} ~~\forall\mu\in
B(\alpha_{k},\zeta_{k}+2\epsilon_{k}).$

\noindent From (\ref{formula4}) we deduce the existence of $n_{k}$
and a $(\delta^{*},n_{k},\epsilon^*)$-separated subset
$\Gamma_{k}\subseteq X_{n_{k},B(\alpha_{k},\zeta_{k})}\subseteq
X_{n_{k},\Xi^{-1}B(\alpha''_{k},\zeta''_{k})}$ with

\begin{eqnarray}\label{formula5}
|\Gamma_{k}|\geqslant\exp\left(n_{k}\left(h(T,\alpha_{k})-\frac{\eta}{2}\right)\right).
\end{eqnarray}

\noindent Assume that $n_{k}$ satisfies

\begin{eqnarray}\label{formula6}
\delta^{*}n_{k}>2g(n_{k})+1~ ~{\rm and}~~
\frac{g(n_{k})}{n_{k}}\leqslant\epsilon_{k}
\end{eqnarray}

\noindent The orbit-segments
$\{x,Tx,\cdots,T^{n_{k}-1}x\},x\in\Gamma_{k},$ are the
building-blocks for the construction of the points of $G.$ By
Proposition \ref{prop2.1} and (\ref{formula6}), we obtain

\begin{eqnarray}\label{formula7}
\begin{split}
&x\in\Gamma_{k}~{\rm and}~y\in B_{n_{k}}(g;x,\epsilon_{k})\\&
\Rightarrow \rho( \alpha_{k}, L_{n_{k}}y)\leq
\zeta_{k}+2\epsilon_{k}\\& \Rightarrow d'( \alpha''_{k},\Xi
L_{n_{k}}y)\leq V(\Xi,\zeta_{k}+2\epsilon_{k}).
\end{split}
\end{eqnarray}

\noindent Choose a strictly increasing sequence $\{N_{k}\},$ with
$N_{k}\in\mathbb{N},$ such that

\begin{eqnarray}\label{formula8}
n_{k+1}\leq \zeta_{k}\sum\limits_{j=1}^{k}n_{j}N_{j}
\end{eqnarray}
and

\begin{eqnarray}\label{formula9}
\sum\limits_{j=1}^{k-1}n_{j}N_{j}\leq\zeta_{k}\sum\limits_{j=1}^{k}n_{j}N_{j}.
\end{eqnarray}

\noindent Finally define the (stretched) sequences
$\{n_{j}'\},\{\epsilon_{j}'\}$ and $\{\Gamma_{j}'\},$ by setting

\begin{equation*}
n_{j}':=n_{k}~~\epsilon_{j}':=\epsilon_{k}~~\Gamma_{j}':=\Gamma_{k},
\end{equation*}
for $j=N_{1}+\cdots+N_{k-1}+q$ with $1\leq q\leq N_{k}.$

\begin{equation*}
G_{k}:=\bigcap\limits_{j=1}^{k}(\bigcup\limits_{x_{j}\in\Gamma_{j}'}T^{-M_{j-1}}B_{n_{j}'}(g;x_{j},\epsilon_{j}')),
\end{equation*}
where $M_{j}:=\sum\limits_{l=1}^{j}n_{l}'.$
  $G_{k}$ is a non-empty closed set. Label each set obtained
by developing this formula by the branches of a labeled tree of
height $k.$ A branch is labeled  by $(x_{1},\cdots,x_{k}),$ with
$x_{j}\in\Gamma'_{j}.$ Theorem \ref{thm3.2} is proved by proving
Lemma \ref{lem3.2}.
\end{proof}

\begin{lem}\label{lem3.2}
Let $\epsilon$ be such that $4\epsilon=\epsilon^{*},$ and let
$$G:=\bigcap\limits_{k\geqslant1}G_{k}.$$

\noindent {\rm (i)} Let $x_{j},y_{j}\in \Gamma_{j}'$ with $x_{j}\neq
y_{j}.$ If $x\in B_{n_{j}'}(g;x_{j},\epsilon_{j}')$ and $y\in
B_{n_{j}'}(g;y_{j},\epsilon_{j}'),$ then

\begin{equation*}
\max\{d(T^{m}x,T^{m}y):m=0,\cdots,n'_{j}-1\}>2\epsilon.
\end{equation*}

\noindent {\rm (ii)} $G$ is a closed set, which is the disjoint
union of non-empty closed  sets $G(x_{1},x_{2},\cdots)$ labeled by
$(x_{1},x_{2},\cdots)$ with $x_{j}\in\Gamma_{j}'.$ Two different
sequences label two different sets.

\noindent {\rm (iii)} $G\subset G_{K'}^{*}.$

\noindent {\rm (iv)} $P(G,\varphi)\geq h^{*}.$
\end{lem}
\begin{proof}
 (i) and (ii) can be seen in \cite{PfiSul2} for details.

\noindent (iii) Define the stretched sequence $\{\alpha_{m}'\}$ by
$\alpha_{m}':=\alpha_{k}$ if
$\sum\limits_{j=1}^{k-1}n_{j}N_{j}+1\leq m\leq
\sum\limits_{j=1}^{k}n_{j}N_{j}.$ The sequence  $\{\alpha_{m}'\}$
has the same limit-point set as the sequence
$\{\alpha_{k}\},\{\alpha_{m}''\}$ has the same limit-point set as
the sequence $\{\Xi\alpha_{k}\}.$ If
$$\lim\limits_{n\rightarrow\infty}d'(\Xi L_{n}y,\alpha''_{n})=0,$$ then the two sequence
$\{\Xi L_{n}y\}$ and $\{\alpha''_{n}\}$ have the same limit-point
set. Because of (\ref{formula8}) and the definition of
$\{\alpha_{m}''\},$ it is sufficient to show that

\begin{equation*}
\lim\limits_{k\rightarrow\infty}d'(\Xi
L_{M_{k}}(y),\alpha_{M_{k}}'')=0.
\end{equation*}
Suppose that
$\sum\limits_{l=1}^{j}n_{l}N_{l}<M_{k}\leq\sum\limits_{l=1}^{j+1}n_{l}N_{l};$
hence $\alpha_{M_{k}}'=\alpha_{j+1}. M_{k}$ can be written as
$M_{k}=\sum\limits_{l=1}^{j-1}n_{l}N_{l}+n_{j}N_{j}+qn_{j+1},$ where
$1\leq q\leq N_{j+1}.$

Since
\begin{eqnarray*}\begin{split}
\Xi &
(L_{\sum\limits_{l=1}^{j-1}n_{l}N_{l}+n_{j}N_{j}+qn_{j+1}}(y))\\&=
\Xi(\frac{1}{M_{k}}(\sum\limits_{l=1}^{j-1}n_{l}N_{l}L_{\sum\limits_{l=1}^{j-1}n_{l}N_{l}}(y)
+\sum\limits_{i=0}^{N_{j}-1}L_{n_{j}}(T^{\sum_{l=1}^{j-1}n_{l}N_{l}+in_{j}}y)+\sum_{i=0}^{q-1}L_{n_{j+1}}(T^{\sum_{l=1}^{j}n_{l}N_{l}+in_{j+1}}y)))
\\&=\frac{\sum\limits_{l=1}^{j-1}n_{l}N_{l}}{M_{k}}\Xi(L_{\sum\limits_{l=1}^{j-1}n_{l}N_{l}}(y))+\frac{\sum\limits_{i=0}^{N_{j}-1}\Xi L_{n_{j}}(T^{\sum\limits_{l=1}^{j-1}n_{l}N_{l}+in_{j}}y)}{M_{k}}
+\frac{\sum\limits_{i=0}^{q-1}\Xi
L_{n_{j+1}}(T^{{\sum\limits_{l=1}^{j}n_{l}N_{l}+in_{j+1}}}y)}{M_{k}}
\end{split}
\end{eqnarray*}
and

\begin{equation*}
\alpha''_{M_{k}}=\frac{\sum\limits_{l=1}^{j-1}n_{l}N_{l}}{M_{k}}\alpha''_{M_{k}}+\frac{\sum\limits_{i=0}^{N_{j}-1}\alpha''_{M_{k}}}{M_{k}}
+\frac{\sum\limits_{i=0}^{q-1}\alpha''_{M_{k}}}{M_{k}},
\end{equation*}
we have

\begin{eqnarray*}\begin{split}
&d'(\Xi(L_{\sum\limits_{l=1}^{j-1}n_{l}N_{l}+n_{j}N_{j}+qn_{j+1}}(y)),\alpha''_{M_{k}})\\&
\leq
\frac{\sum\limits_{l=1}^{j-1}n_{l}N_{l}}{M_{k}}d'(\Xi(L_{\sum\limits_{l=1}^{j-1}n_{l}N_{l}}(y)),\alpha''_{M_{k}})+\frac{\sum\limits_{i=0}^{N_{j}-1}d'(\Xi
L_{n_{j}}(T^{\sum\limits_{l=1}^{j-1}n_{l}N_{l}+in_{j}}y),\Xi\alpha_{j+1})}{M_{k}}\\&+\frac{\sum\limits_{i=0}^{q-1}d'(\Xi
L_{n_{j+1}}(T^{{\sum\limits_{l=1}^{j}n_{l}N_{l}+in_{j+1}}}y),\Xi\alpha_{j+1})}{M_{k}}\\&\leq\zeta_{j}V(\Xi,1)+d'(\Xi\alpha_{j},\Xi\alpha_{j+1})+V(\Xi,\zeta_{j}+2\epsilon_{j})+V(\Xi,\zeta_{j+1}+2\epsilon_{j+1}).
\end{split}
\end{eqnarray*}
Since
$\lim\limits_{j}\zeta_{j}=0,\lim\limits_{j}\epsilon_{j}=0\Rightarrow
\lim\limits_{j}V(\Xi,\epsilon_{j})=0~~{\rm and}
~\lim\limits_{j}d'(\Xi\alpha_{j},\Xi\alpha_{j+1})=0$ this proves
(iii).

\noindent (iv) From the choice of $\{N_{k}\}$ we can get
$\lim\limits_{n\rightarrow\infty}\frac{M_{n}}{M_{n+1}}=1,$ where
$M_{j}=n'_{1}+\cdots+n'_{j}.$ There exist $n_{k}\in\mathbb{N}$ and a
$(\delta^{*},n_{k},\epsilon^{*})$-separated subset $\Gamma_{k}$ of
$X_{n,B(\alpha_{k},\zeta_{k})}$ such that $$\#\Gamma_{k}\geq
\exp(n_{k}(h(T,\alpha_{k})-\eta/2)).$$ And for any
$^{k}x\in\Gamma_{k},$ we have $L_{n}(^{k}x)\in
B(\alpha_{k},\zeta_{k}).$ So $$|\int\varphi
dL_{n}(^{k}x)-\int\varphi
d\alpha_{k}|=|\frac{1}{n}S_{n}\varphi(^{k}x)-\int\varphi
d\alpha_{k}|\leq\frac{\delta}{6}.$$ Thus

\begin{eqnarray*}\begin{split}
\#\Gamma_{k}&\geq\exp(n_{k}(h(T,\alpha_{k})+\int \varphi
d\alpha_{k}-\eta/2)-S_{n_{k}}\varphi(^{k}x)-n_{k}\frac{\delta}{6})\\
&\geq
\exp(n_{k}h^{*}-S_{n_{k}}\varphi(^{k}x)-n_{k}\frac{\delta}{6}).
\end{split}
\end{eqnarray*}
Since $G$ is a compact set we can just consider finite covers
$\mathcal{C}$ of $G$ with the property that if
$B_{m}(x,\epsilon)\in\mathcal{C},$ then $B_{m}(x,\epsilon)\cap
G\neq\emptyset, \forall B_{m}(x,\epsilon)\in \mathcal{C}.$ For each
 $\mathcal{C}\in\Gamma_{n}(G,\epsilon)$ we define the
cover $\mathcal{C}^{'},$ in which each ball $B_{m}(x,\epsilon)$ is
replaced by  $B_{M_{p}}(x,\epsilon)$ when  $M_{p}\leq m< M_{p+1}.$
Then
\begin{eqnarray*}M(G,s,\varphi,n,\epsilon)&=&\inf_{\mathcal{C}\in
\Gamma_{n}(G,\epsilon)}\sum_{B_{m}(x,\epsilon)\in
\mathcal{C}}\exp(-s m+\sup_{y\in
B_{m}(x,\epsilon)}S_{m}\varphi(y))\\
&\geq& \inf_{\mathcal{C}\in
\Gamma_{n}(G,\epsilon)}\sum_{\substack{B_{M_{p}}(x,\epsilon)\in
\mathcal{C}^{'},\\ {z\in B_{m}(x,\epsilon)\cap
B_{M_{p}}(x,\epsilon)\cap
G}}}\exp(-sm+S_{m}\varphi(z))\end{eqnarray*}

 \noindent Consider a specific $\mathcal{C}^{'}$ and let $m$
be the largest value of  $p$ such that there exists
$B_{M_{p}}(x,\epsilon)\in \mathcal{C}^{'}.$

\noindent Set$$\mathcal{W}_{k}:=\prod_{i=1}^{k}\Gamma_{k}^{'},
\mathcal{\overline{W}}_{m}:=\bigcup_{k=1}^{m}\mathcal{W}_{k}.$$

\noindent Each $z\in B_{M_{p}}(x,\epsilon)\cap G$ corresponds to a
point in $\mathcal{W}_{p}.$ Lemma \ref{lem3.2} (i) implies that this
point is uniquely defined. For $1\leq j\leq k,$ the word
$v=(v_{1},\dots,v_{j})\in \mathcal{W}_{j}$ is a prefix of
$w=(w_{1},\dots,w_{k})\in \mathcal{W}_{k},$ if
$v_{i}=w_{i},i=1,\dots,j.$ Note that each $w\in\mathcal{W}_{k}$ is
the prefix of exactly $|\mathcal{W}_{m}|/|\mathcal{W}_{k}|$ words of
$\mathcal{W}_{m}.$ If $\mathcal{W}\subset
\overline{\mathcal{W}}_{m}$ contains a prefix of each word of
$\mathcal{W}_{m},$ then

\begin{equation*}
\sum_{k=1}^{m}|\mathcal{W}\cap\mathcal{W}_{k}||\mathcal{W}_{m}|
/|\mathcal{W}_{k}|\geq|\mathcal{W}_{m}|.
\end{equation*}
So if $\mathcal{W}$ contains a prefix of each word of
$\mathcal{W}_{m},$ then

\begin{equation*}
\sum_{k=1}^{m}|\mathcal{W}\cap\mathcal{W}_{k}|
/|\mathcal{W}_{k}|\geq1.
\end{equation*}
Since $\mathcal{C}^{'}$ is a cover, each point of $\mathcal{W}_{m}$
has a prefix associated with some
$B_{M_{p}}(x,\epsilon)\in\mathcal{C}^{'}.$ Hence,

\begin{equation*}
|\mathcal{W}_{p}|\geq\exp[M_{p}h^{*}-\sum_{i=1}^{p}(S_{n_{i}^{'}}\varphi(^{i'}x)+n_{i}^{'}\delta/6)],
\end{equation*}
where $^{i'}x\in \Gamma_{i}^{'}.$
 So $$\sum_{B_{M_{p}}(x,\epsilon)\in
\mathcal{C}^{'}}\exp[-M_{p}h^{*}+\sum_{i=1}^{p}(S_{n_{i}^{'}}\varphi(^{i'}x)+n_{i}^{'}\delta/6)]\geq
1.$$

\noindent Next, we want to prove
$M_{p}h^{*}-\sum\limits_{i=1}^{p}(S_{n_{i}^{'}}\varphi(^{i'}x)+n_{i}^{'}\delta/6)-sm+S_{m}\varphi(z)
=m(h^{*}-s)+\sum\limits_{i=1}^{p}(S_{n_{i}^{'}}\varphi(T^{M_{i-1}}z)-S_{n_{i}^{'}}\varphi(^{i'}x)-n_{i}^{'}\delta/6)
+S_{m-M_{p}}\varphi{(T^{M_{p}}z)} -(m-M_{p})h^{*}>0.$
 Since $z\in G,$ by the construction of $G$ we know
there exists a close subset
$$G(x_{1},x_{2},\cdots)=\bigcap_{i=0}^{\infty}T^{-M_{j-1}}B_{n_{j}^{'}}(g;x_{j},\epsilon_{j}^{'}),$$
such that $T^{M_{j-1}}z\in B_{n_{j}^{'}}(g;x_{j},\epsilon_{j}^{'}).$

\noindent By (\ref{formula7}) and $^{i'}x\in\Gamma_i'$ we get
$L_{n_{i}^{'}}(T^{M_{i-1}}z)\in
B(\alpha_{i}^{'},\zeta_{i}^{'}+2\epsilon_{i}^{'})$ and
$L_{n_{i}^{'}}(^{i'}x)\in B(\alpha_{i}^{'},\zeta_{i}^{'}).$ Thus,

\begin{equation*}
\Big|\int\varphi dL_{n_{i}^{'}}(T^{M_{i-1}}z)-\int\varphi
dL_{n_{i}^{'}}(^{i'}x)\Big|
n_{i}^{'}=\Big|S_{n_{i}^{'}}\varphi(T^{M_{i-1}}z)-S_{n_{i}^{'}}\varphi(^{i'}x)\Big|
\leq n_{i}^{'}\delta/2.
\end{equation*}
So,
\begin{eqnarray*}\begin{split}
&M_{p}h^{*}-\sum_{i=1}^{p}(S_{n_{i}^{'}}\varphi(x^{i'})+n_{i}^{'}\delta/6)-sm+S_{m}\varphi(z)\\
 &\geq
m(h^{*}-s)-\sum_{i=1}^{p}2n_{i}^{'}\delta/3-n_{p+1}^{'}(\parallel\varphi\parallel+h^{*})
\\
&\geq 2\delta
M_{p}-M_{p}\delta-n_{p+1}^{'}(\parallel\varphi\parallel+h^{*})\\
&\geq M_{p}\delta-n_{p+1}^{'}(\parallel\varphi\parallel+h^{*}).
\end{split}
\end{eqnarray*}
Since $\lim_{p\to\infty}\frac{n_{p+1}^{'}}{M_{p}}=0,$ it is possible
to choose sufficient large $p$ such that
$M_{p}\delta-n_{p+1}^{'}(\parallel\varphi\parallel+h^{*})>0.$ Then

\begin{equation*}
\sum_{B_{m}(x,\epsilon)\in \mathcal{C}}\exp(-s m+\sup_{y\in
B_{m}(x,\epsilon)}S_{m}\varphi(y))
 \geq\sum_{B_{M_{p}}(x,\epsilon)\in
\mathcal{C}^{'}}\exp[-M_{p}h^{*}+\sum_{i=1}^{p}(S_{n_{i}^{'}}\varphi(^{i'}x)+n_{i}^{'}\delta/6)].
\end{equation*}
It implies  $M(G,s,\varphi,n,\epsilon)\geq1,$ i.e., $s\leq
P(G,\varphi,\epsilon).$ Letting $s\to h^{*},$ we complete the proof
of Lemma \ref{lem3.2}.
\end{proof}

\begin{rem}
The quintuple $(X,T,Y,\Xi,L_{n})$ satisfying g-almost product
property and the uniform separated condition means:

\noindent (a) $X$ is a compact metric space, $T:X\rightarrow X$ is a
continuous map satisfying g-almost product property and the uniform
separated condition.

\noindent (b) $Y$ a vector space, $\Xi:M(X)\rightarrow Y$ is a
continuous and affine map.

\noindent (c)  $L_{n}x:X\rightarrow M(X),$ where
$L_{n}x=\sum\limits_{i=0}^{n-1}\delta_{T^{i}x}.$
\end{rem}
\noindent For $y\in Y,$ set

\begin{equation*}
\Delta(y)=\{x\in X:\{y\}=A(\Xi
L_{n}x)\},\widetilde{\Delta}(y)=\{x\in X:y\in A(\Xi L_{n}x)\}.
\end{equation*}
It is easy to get the following corollary by the above two theorems.

\begin{cro}\label{cor3.1}
For any $\varphi\in C(X,\mathbb{R}),$ if $(X,T,Y,\Xi,L_{n})$
satisfies g-almost product property and the uniform separated
condition, then

\begin{equation*}
P(\Delta(y),\varphi)=P(\widetilde{\Delta}(y),\varphi).
\end{equation*}
\end{cro}

\section{Proof of Theorems}
This section is aim to prove the theorems.

\noindent {\it Proof  of Theorem \ref{thm1.1}}

\noindent (1)It can be obtained by proposition \ref{prop2.5}.

\noindent (2) $C\subset Y$ is a compact and connected subset of
$\Xi(M(X,T)),$ then
$P(\Delta_{equ}(C),\varphi)=P(G_{C}^{*},\varphi)=\inf\limits_{y\in
C}\Lambda(y,\varphi).$

\noindent {\it Proof of Theorem \ref{thm1.2}}

\noindent (1) It is obvious.

\noindent (2) We prove it by presenting several lemmas.

\begin{lem}\label{lem4.1}
Let $(X,T,Y,\Xi,L_{n})$ satisfy g-almost product property and the
uniform separated condition. If $C\subset Y$ and $\varphi\in
C(X,\mathbb{R}),$ then

\begin{equation*}
\inf_{y\in C}\Lambda(y,\varphi)=\inf_{y\in co(C)}\Lambda(y,\varphi),
\end{equation*}
where $co(C)$ is convex hull of $C.$
\end{lem}

\begin{proof} The direction $\geq$ is obvious.
 As to the
other direction, for any $y\in co(C),$ fix $\epsilon>0.$ Since $y\in
co(C),$ there exist $y_{1},y_{2},\cdots,y_{n}\in C$ and
$\lambda_{1},\cdots,\lambda_{n}\geq 0$ with
$\sum\limits_{i}\lambda_{i}=1$ such that
$\sum\limits_{i}\lambda_{i}y_{i}=y.$ For each $y_{i},$ choose
$\mu_{i}\in M(X,T)$ s.t. $\Xi\mu_{i}=y_{i}$ and $h(T,\mu_{i})+\int
\varphi d\mu_{i}\geq\Lambda(y_{i},\varphi)-\epsilon.$ Since the
entropy function is affine and $\sum\limits_{i}\lambda_{i}\mu_{i}\in
M(X,T)$ satisfies
$\Xi(\sum_{i}\lambda_{i}\mu_{i})=\sum_{i}\lambda_{i}\Xi\mu_{i}=\sum_{i}\lambda_{i}y_{i}=y,$
we get
\begin{eqnarray*}\begin{split}
\Lambda(y,\varphi)&=\sup_{\mu\in
M(X,T),\Xi\mu=y}\{h(T,\mu)+\int\varphi d\mu\}\\& \geq
h(T,\sum_{i}\lambda_{i}\mu_{i})+\int \varphi
d(\sum\limits_{i}\lambda_{i}\mu_{i})\\&=\sum_{i}\lambda_{i}(h(T,\mu_{i})+\int\varphi
d\mu_{i})\\&\geq\sum_{i}\lambda_{i}\Lambda(y_{i},\varphi)-\epsilon\\&
\geq \inf_{y\in C}\Lambda(y,\varphi)-\epsilon.
\end{split}
\end{eqnarray*}
Thus,

\begin{equation*}
 \inf_{y\in C}\Lambda(y,\varphi)\leq\inf_{y\in
co(C)}\Lambda(y,\varphi).
\end{equation*}
\end{proof}

\begin{lem}\label{lem4.2}
Let $(X,T,Y,\Xi,L_{n})$ satisfy g-almost product property and the
uniform separated condition. If $C\subset Y$ and $\varphi\in
C(X,\mathbb{R}),$
 then
$$\inf_{y\in C}
\Lambda(y,\varphi)=\inf_{y\in\overline{C}}\Lambda(y,\varphi).$$
\end{lem}
\begin{proof} $\geq$ is obvious. The other direction follows from the fact
that $y\to\Lambda(y,\varphi)$ is upper semi-continuous. C. Pfister
and W. Sullivan \cite{PfiSul2} proved that the entropy map on
$M(X,T),\mu\rightarrow h(T,\mu),$ is upper semi-continuous under the
g-almost product property and uniformly separation property.
$\forall\gamma>0,\forall y\in\overline{C},\exists \{y_{n}\}\subseteq
C,$ s.t. $ y_{n}\rightarrow y,$ as $n\to \infty$ and there exists $
\mu_{n}\in M(X,T)\bigcap \Xi^{-1}y_{n},$ s.t.
$\Lambda(y_{n},\varphi)\leq h(T,\mu_{n})+\int \varphi
d\mu_{n}+\gamma .$ Assume that $\mu$ is a limit-point of
$\{\mu_{n}\},$ then $\Xi \mu=y.$ So

\begin{eqnarray*}\begin{split}
\limsup\limits_{n\rightarrow\infty}\Lambda(y_{n},\varphi)&\leq
\limsup\limits_{n\rightarrow\infty}h(T,\mu_{n})+\int \varphi
d\mu_{n}+\gamma\\&\leq h(T,\mu)+\int\varphi
d\mu+3\gamma\\&\leq\Lambda(y,\varphi)+3\gamma.
\end{split}
\end{eqnarray*}
The conclusion of Lemma \ref{lem4.2} follows.

\noindent Now, continue the proof of (2).  It suffices to show for
any nonempty $C\subset \Xi(M(X,T)),$

\begin{equation*}
P(\Delta_{equ}(\overline{co(C)}),\varphi)=P(\Delta_{sub}(C),\varphi).
\end{equation*}
Since $\Delta_{equ}(\overline{co(C)})=\{x\in X|A(\Xi
L_{n}x)=\overline{co(C)}\}\subset\{x\in X|C\subset A(\Xi
L_{n}x)\}=\Delta_{sub}(C),$ it is obvious that
$P(\Delta_{sub}(C),\varphi)\geq
P(\Delta_{equ}(\overline{co(C)}),\varphi).$  On the other hand, by
Corollary \ref{cor3.1}, if $\Delta(y)\neq\emptyset,$ then
$P(\Delta(y),\varphi)=P(\widetilde{\Delta}(y),\varphi).$ So for any
$y\in C,$

\begin{eqnarray*}\begin{split}
P(\Delta_{sub}(C),\varphi)&\leq P(\{x\in X|\{y\}\subset A(\Xi
L_{n}x)\},\varphi)\\&=P(\widetilde{\Delta}(y),\varphi)\\&=P(\Delta(y),\varphi)\\&=\Lambda(y,\varphi).
\end{split}
\end{eqnarray*}
Hence,

\begin{eqnarray*}\begin{split}
P(\Delta_{sub}(C),\varphi)&\leq \inf_{y\in C}\Lambda(y,\varphi)\\&=
\inf_{y\in
\overline{co(C)}}\Lambda(y,\varphi)\\&=P(\Delta_{equ}(\overline{co(C)}),\varphi).
\end{split}
\end{eqnarray*}
So,
\begin{equation*}
P(\Delta_{sub}(C),\varphi)=\inf\limits_{y\in C}\sup\limits_{\mu\in
M(X,T)\atop\Xi\mu=y}\{h(T,\mu)+\int\varphi d\mu\}=\inf\limits_{y\in
C} \Lambda(y,\varphi).
\end{equation*}
\end{proof}
\noindent {\it Proof  of Theorem \ref{thm1.3}}

\noindent First, we show the following proposition.
\begin{prop}\label{prop4.1}
 $(X,T,Y,\Xi,L_{n})$ as before and
$\varphi\in C(X,\mathbb{R}).$ If $C\subset Y,$  then

\begin{equation*}
P(\Delta_{sup}(C),\varphi)=P(\Delta_{sup}(C\cap
\Xi(M(X,T))),\varphi)=\sup_{y\in C}\Lambda(y,\varphi).
\end{equation*}
\end{prop}

\begin{proof}
 $P(\Delta_{sup}(C),\varphi)\leq P(^{C}G^{*},\varphi)\leq
\sup_{y\in C}\Lambda(y,\varphi).$

\noindent On the other hand, $\forall \epsilon
>0,\exists y'\in C\cap\Xi(M(X,T)),$ s.t. $\sup_{y\in C}\Lambda(y,\varphi)\leq
\Lambda(y',\varphi)+\epsilon,$ and $\Lambda(y',\varphi)=
P(G_{\{y'\}}^{*},\varphi)\leq P(\Delta_{sup}(C),\varphi).$

\noindent So,

\begin{equation*}
P(\Delta_{sup}(C),\varphi)+\epsilon\geq\sup_{y\in
C}\Lambda(y,\varphi).
\end{equation*}
Thus,

\begin{equation*}
P(\Delta_{sup}(C),\varphi)=\sup_{y\in C}\Lambda(y,\varphi).
\end{equation*}
\end{proof}
Now, We continue the proof of Theorem \ref{thm1.3}.

\noindent (1) It comes from Proposition \ref{prop4.1}.

\noindent (2) Given $S_{1}\subseteq Q\subseteq
S_{2},Q\subseteq\Xi(M(X,T))$
 is compact and connected.

 \noindent Since $\Delta_{equ}(Q)=\{x\in X|A(\Xi L_{n}x)=Q\}\subset\{S_{1}\subset A(\Xi L_{n}x)\subset
 S_{2}\},$ we get

\begin{equation*}
P(\Delta(S_{1},S_{2}),\varphi)\geq P
(\Delta_{equ}(Q),\varphi)=\inf_{x\in
 Q}\Lambda(x,\varphi).
 \end{equation*}
Since $Q$ is arbitrary, we obtain

\begin{equation*}
 P(\Delta(S_{1},S_{2}),\varphi)\geq \sup_{S_{1}\subset
Q\subset S_{2}\atop Q\subseteq\Xi(M(X,T)) ~is ~ compact~and ~
connected }\inf_{x\in
 Q}\Lambda(x,\varphi).
 \end{equation*}
As to the other inequality, observe
$$\Delta(S_{1},S_{2})\subset\Delta(S_{1},Y)=\Delta_{sub}(S_{1}).$$

\noindent (3)  It is obvious.

\noindent {\it Proof  of Theorem \ref{thm1.4}}

\noindent (1) It follows from Proposition \ref{prop4.1}.

\noindent (2) Combining the fact $S_{1}\subseteq
\overline{co}(S_{1})$ and $\overline{co}(S_{1})$ is a compact and
connected subset of $S_{2},$ and Theorem \ref{thm1.3}, we obtain
\begin{eqnarray*}\begin{split}
\inf_{y\in
S_{1}}\Lambda(y,\varphi)&=\inf_{y\in\overline{co}(S_{1})}\Lambda(y,\varphi)\\&\leq\sup_{S_{1}\subseteq
Q\subseteq S_{2}\atop Q~is ~compact~and~connected}\inf_{y\in
Q}\Lambda(y,\varphi)\\&\leq
P(\Delta(S_{1},S_{2}),\varphi)\\&\leq\inf_{y\in
S_{1}}\Lambda(y,\varphi).
\end{split}
\end{eqnarray*}

\noindent (3) It is obvious.

\section{Some applications}
In the section,  firstly, we present some spetra induced by
different deformations. Secondly, we use BS-dimension to describe
some level sets. Thirdly, the relative multifractal spectrum of
ergodic averages are discussed. At last, symbolic space and iterated
function systems are investigated.
\subsection{Some spectra}
 Different spectra are induced by different deformations
$(Y,\Xi)\cite{Ols2}.$

\begin{itemize}
  \item The spectrum of the historic set  of ergodic
  averages. Let $\varphi:X\rightarrow\mathbb{R}$ be continuous and
  define $\Xi:M(X)\rightarrow\mathbb{R}$ by $\Xi:\mu\mapsto\int\varphi
  d\mu.$ In this case we obtain for $S_{1},S_{2}\subset\mathbb{R},$

\begin{equation*}
\Delta(S_{1},S_{2})=\left\{x\in X:S_{1}\subset
A\left(\frac{1}{n}\sum_{k=0}^{n-1}\varphi(T^{k}x)\right)\subset
  S_{2}\right\}.
\end{equation*}

  \item The spectrum of the historic set of  empirical
  measures. Define $\Xi:M(X)\rightarrow M(X) $by
  $$\Xi:\mu\rightarrow\mu.$$ In this case we obtain for $S_{1},S_{2}\subset
  \mathbb{R},$

\begin{equation*}
\Delta(S_{1},S_{2})=\{x\in X:S_{1}\subset
A(\frac{1}{n}\sum_{k=0}^{n-1}\delta_{T^{k}x})\subset S_{2}\}.
\end{equation*}

  \item The spectrum of the historic set of   local
  Lyapunov exponents. Let $X$ be a differentiable manifold and
$T:X\to
  X$ be a $C^{1}$ map. The local  Lyapunov exponents of $T$ at the
  point $x$ is defined by $\chi(x)=\lim_{n\rightarrow\infty}\frac{1}{n}\log|(DT^{n})(x)|.$
Define $\Xi:M(X)\rightarrow\mathbb{R}$ by $$\mu\rightarrow\int
DTd\mu.$$ In this case we obtain for $S_{1},S_{2}\subseteq
\mathbb{R},$

\begin{equation*}
\Delta(S_{1},S_{2})=\{x\in X:S_{1}\subset
A(\frac{1}{n}\log|(DT^{n})(x)|)\subset S_{2}\}.
\end{equation*}

\item The mixed spectrum of the historic set  of ergodic
averages of arbitrary families of continuous functions. Assume that
the family of maps
$(M(X)\rightarrow\mathbb{R}:\mu\mapsto\int\varphi_{i}d\mu)_{i\in I}$
is totally bounded. Define $\Xi:M(X)\rightarrow l^{\infty}(I)$ by
$$\Xi:\mu\mapsto(\int \varphi_{i}d\mu)_{i\in I}.$$ In this case
we obtain for $S_{1},S_{2}\subset
l^{\infty}(I),$$$\Delta(S_{1},S_{2})=\{x\in X:S_{1}\subset
A((\frac{1}{n}\sum_{k=0}^{n-1}\varphi_{i}(T^{k}x))_{i\in I})\subset
S_{2}\}$$
\end{itemize}
We only consider some (not all) spectrum above and obtain several
corollaries as examples.  It is  easy to get Corollary \ref{cor5.1}
and Corollary \ref{cor5.3}. So we omit the proof.

\begin{cro}\label{cor5.1}
 $(X,T,L_{n})$ as before. Let $Y=\mathbb{R}$ and
$\phi_{j}:X\rightarrow \mathbb{R}$
 be a family of continuous functions. Assume the family of maps $(\Xi_{j}:M(X)\rightarrow\mathbb{R}:\mu\mapsto\int \phi_{j}d\mu)_{j\in
 I}$ is totally bounded. Fix $S_{1},S_{2}\subset l^{\infty}(I),\psi\in C(X,\mathbb{R}).$
 \begin{enumerate}
     \item
     If $S_{1}=\emptyset$ and $S_{2}$ is closed and convex, then
\begin{equation*}
     P(\{x\in X:A((\frac{1}{n}\sum_{k=0}^{n-1}\phi_{j}(T^{k}x))_{j\in I})\}\subset S_{2} ,\psi)=\sup_{x\in S_{2}}\Lambda(x,\psi).
\end{equation*}
      \item
     If $S_{1}\neq\emptyset$ and $\overline{co}(S_{1})$ is contained in a connected component of $S_{2},$ then
\begin{equation*}
     P(\{x\in X:S_{1}\subset A((\frac{1}{n}\sum_{k=0}^{n-1}\phi_{j}(T^{k}x))_{j\in I})\},\psi)=\inf_{x\in S_{1}}\Lambda(x,\psi).
\end{equation*}
        \item
     If $S_{1}\neq\emptyset$ and $\overline{co}(S_{1})$ is not contained in a connected component of $S_{2},$ then
\begin{equation*}
  \{x\in X:S_{1}\subset A((\frac{1}{n}\sum_{k=0}^{n-1}\phi_{j}(T^{k}x))_{j\in I})\}=\emptyset.
\end{equation*}
  \end{enumerate}
\end{cro}

\begin{cro}\label{cor5.3}
$(X,T,L_{n})$ as before. Let $(Y_{i},\Xi_{i})_{i}$ be (a possible
uncountable)
 family of deformations and assume that $Y_{i}$ is a normed vector space and that $\Xi_{i}:M(X)\rightarrow Y_{i}$ is affine and
 continuous. Define the vector spaces $\times_{i}Y_{i}$ and
 $[\times_{i}Y_{i}]^{\infty}$ by $$\times_{i}Y_{i}=\{(y_{i})_{i}:y_{i}\in Y_{i}\forall
 i\},$$$$[\times_{i}Y_{i}]^{\infty}=\{(y_{i})_{i}\in\times_{i}Y_{i}:\sup_{i}||y_{i}||<\infty\},$$
 and equip $[\times_{i}Y_{i}]^{\infty}$ with the norm
 $||(y_{i})_{i}||=\sup\limits_{i}||y_{i}||.$ Assume $\sup\limits_{\mu\in
 M(X),i}||\Xi_{i}\mu||<\infty$ and the map
 $$M(X)\rightarrow[\times_{i}Y_{i}]^{\infty}:\mu\mapsto(\Xi_{i}\mu)_{i}$$
 is continuous, $\Xi=(\Xi_{i})_{i\in I}.$ Fix $S_{1},S_{2}\subset[\times_{i}Y_{i}]^{\infty}, \psi\in C(X,\mathbb{R})$
 \begin{enumerate}
     \item
     If $S_{1}=\emptyset$ and $S_{2}$ is closed and convex, then
\begin{equation*}
    P(\{x\in X:S_{1}\subset A((\Xi_{j}L_{n}x)_{j\in I})\subset S_{2}\},\psi)=\sup_{x\in S_{2}}\sup\limits_{\mu\in M(X,T)\atop(\Xi_{j}\mu)_{j\in I}=x}\left\{h(T,\mu)+\int\psi
    d\mu\right\}.
\end{equation*}
      \item
     If $S_{1}\neq\emptyset$ and $\overline{co}(S_{1})$ is contained in a connected component of $S_{2},$ then
\begin{equation*}
    P(\{x\in X:S_{1}\subset A((\Xi_{j}L_{n}x)_{j\in I})\subset S_{2}\},\psi)=\inf_{x\in S_{1}}\sup\limits_{\mu\in M(X,T)\atop(\Xi_{j}\mu)_{j\in I}=x}\left\{h(T,\mu)+\int\psi
    d\mu\right\}.
\end{equation*}
        \item
     If $S_{1}\neq\emptyset$ and $\overline{co}(S_{1})$ is not contained in a connected component of $S_{2},$ then
\begin{equation*}
    \{x\in X:S_{1}\subset A((\Xi_{j}L_{n}x)_{j\in I})\subset S_{2}\}=\emptyset.
\end{equation*}
 \end{enumerate}
\end{cro}

\noindent Next, we use dimension theory to discuss
$\Delta_{equ}(\cdot), \Delta_{sub}(\cdot)$ and so on.

 Let
$\psi:X\rightarrow\mathbb{R}$ be a strictly positive continuous
function. For each set $Z\subset X$ and each number
$t\in\mathbb{R},$ define
\begin{equation*}
N(Z,t,\psi,n,\epsilon):=\inf\limits_{\mathcal{C}\in\mathcal{G}_{n}(Z,\epsilon)}\left\{\sum\limits
_{B_{m}(x,\epsilon)\in\mathcal{C}}\exp(-t\sup\limits_{y\in
B_{m}(x,\epsilon)}S_{m}\psi(y))\right\},
\end{equation*}
where
$\mathcal{G}_{n}(Z,\epsilon)$ is the collection of all finite or
countable covers of $Z$ by sets of the form $B_{m}(x,\epsilon),$
with $m\geq n.$
$$N(Z,t,\psi,\epsilon)=\lim\limits_{N\rightarrow\infty}N(Z,t,\psi,n,\epsilon),$$
Set
$$BS(Z,\psi,\epsilon)=\inf\{t:N(Z,t,\psi,\epsilon)=0\}=\sup\{t:N(Z,t,\psi,\epsilon)=\infty\}.$$
Let
$BS(Z,\psi)=\lim\limits_{\epsilon\rightarrow0}BS(Z,\psi,\epsilon),$
and we call it the BS-dimension of $Z.$ This notation was introduced
by Barreira and Schmeling \cite{BarSch}.

By the definition of topological pressure and BS-dimension, we can
get that for any set $Z\subset X,$ the BS-dimension of $Z$ is a
unique root of Bowen's equation $P(Z,-s\psi)=0,$ i.e.
$s=BS(Z,\psi).$

The following corollaries are easy to obtain from  above theorems.

\begin{cro}\label{cor5.4}
$(X,T,\Xi,L_{n},Y)$ satisfies g-almost product property and the
uniform separation property and $\varphi\in C(X,\mathbb{R}^{+}).$ If
\begin{enumerate}
  \item $C\subset Y$
is not a compact and connected subset of $\Xi(M(X,T)),$ then

\begin{equation*}
\{x\in X:A(\Xi L_{n}x)=C \}=\emptyset,
\end{equation*}

  \item $C\subset Y$ is a compact and connected subset of $\Xi(M(X,T)),$ then

\begin{equation*}
BS(\Delta_{equ}(C),\varphi)=\inf\limits_{y\in C}\sup\limits_{\mu\in
M(X,T)\atop\Xi \mu=y}\left\{\frac{h(T,\mu)}{\int\varphi
d\mu}\right\}.
\end{equation*}
\end{enumerate}
\end{cro}

\begin{cro}\label{cor5.5}
 $(X,T,\Xi,L_{n},Y)$ as before
and $\varphi\in C(X,\mathbb{R}^{+}).$ If
\begin{enumerate}
  \item  $C\subset
Y$ is not a subset of $\Xi(M(X,T)),$ then

\begin{equation*}
\{x\in X:C\subset A(\Xi L_{n}x)\}=\emptyset.
\end{equation*}

  \item   $C\subset Y$ is  a subset of $\Xi(M(X,T)),$ then

\begin{equation*}
BS(\Delta_{sub}(C),\varphi)=\inf\limits_{y\in C}\sup\limits_{\mu\in
M(X,T)\atop\Xi\mu=y}\left\{\frac{h(T,\mu)}{\int\varphi
d\mu}\right\}.
\end{equation*}
\end{enumerate}
\end{cro}

\begin{cro}\label{cor5.6}
$(X,T,\Xi,L_{n},Y)$ as before and $\varphi\in C(X,\mathbb{R}^{+}),$
fix $S_{1}\subset\Xi(M(X,T)),S_{2}\subset Y,$ if
\begin{enumerate}
  \item $S_{1}=\emptyset,$ then

\begin{equation*}
BS(\Delta(S_{1},S_{2}),\varphi)=\sup\limits_{x\in
S_{2}}\sup\limits_{\Xi\mu=x\atop\mu\in
M(X,T)}\left\{\frac{h(T,\mu)}{\int\varphi d\mu}\right\}.
\end{equation*}

  \item   $S_{1}\neq\emptyset$ and $S_{1}$ is contained in a connected
component of $S_{2},$ then
\begin{eqnarray*}\begin{split}
\sup\limits_{S_{1}\subseteq Q\subseteq
S_{2}\atop~Q\subseteq\Xi(M(X,T))~is~ compact~and~
connected}\inf\limits_{x\in Q}\sup\limits_{\Xi\mu=x\atop\mu\in
M(X,T)}\left\{\frac{h(T,\mu)}{\int\varphi d\mu}\right\}\\ \leq
BS(\Delta(S_{1},S_{2}),\varphi) \leq\inf\limits_{x\in
S_{1}}\sup\limits_{\Xi\mu=x\atop\mu\in
M(X,T)}\left\{\frac{h(T,\mu)}{\int\varphi d\mu}\right\}.
\end{split}
\end{eqnarray*}

  \item  $S_{1}\neq\emptyset$ and $S_{1}$ is not contained in a
connected component of $S_{2},$ then $$\{x\in X:S_{1}\subset A(\Xi
L_{n}x)\subset S_{2}\}=\emptyset.$$
\end{enumerate}
\end{cro}

\begin{cro}\label{cor5.7}
$(X,T,\Xi,L_{n},Y)$ as before and $\varphi\in C(X,\mathbb{R}^{+}),$
fix $S_{1}\subset\Xi(M(X,T)),S_{2}\subseteq Y.$
\begin{enumerate}
  \item  If $S_{1}=\emptyset,$
then
\begin{equation*}
BS(\Delta(S_{1},S_{2}),\varphi)=\sup\limits_{x\in
S_{2}}\sup\limits_{\Xi\mu=x\atop\mu\in
M(X,T)}\left\{\frac{h(T,\mu)}{\int\varphi d\mu}\right\}.
\end{equation*}

  \item If $S_{1}\neq\emptyset$ and $\overline{co}(S_{1})$ the closed convex hull
of $S_{1}$ is contained in a connected component of $S_{2},$ then

\begin{equation*}
BS(\Delta(S_{1},S_{2}),\varphi)=\inf\limits_{x\in
S_{1}}\sup\limits_{\Xi\mu=x\atop\mu\in
M(X,T)}\left\{\frac{h(T,\mu)}{\int\varphi d\mu}\right\}.
\end{equation*}

  \item If $S_{1}\neq\emptyset$ and $S_{1}$ is not contained in a connected
component of $S_{2}$, then $$\{x\in X:S_{1}\subset A(\Xi
L_{n}x)\subset S_{2}\}=\emptyset.$$
\end{enumerate}
\end{cro}

\begin{cro}\label{cor5.8}
 $(X,T,L_{n})$ as before. Let
$Y=\mathbb{R}$ and $\phi_{j}:X\rightarrow \mathbb{R}$
 be a family of continuous functions. Assume the family of maps $(\Xi_{j}:M(X)\rightarrow\mathbb{R}:\mu\mapsto\int \phi_{j}d\mu)_{j\in
 I}$ is totally bounded, $\Xi=(\Xi_{i})_{i\in I}.$ Fix $S_{1},S_{2}\subset l^{\infty}(I),\varphi\in C(X,\mathbb{R}^{+}).$
 \begin{enumerate}
     \item
     If $S_{1}=\emptyset$ and $S_{2}$ is closed and convex, then
\begin{equation*}
BS\left(\left\{x\in
X:A\left(\left(\frac{1}{n}\sum_{k=0}^{n-1}\phi_{j}(T^{k}x)\right)_{j\in
I}\right)\right\}\subset S_{2} ,\varphi\right)
     =\sup_{x\in S_{2}}\sup_{\Xi\mu=x\atop\mu\in M(X,T)}\left\{\frac{h(T,\mu)}{\int\varphi d
     \mu}\right\}.
\end{equation*}

      \item
     If $S_{1}\neq\emptyset$ and $\overline{co}(S_{1})$ is contained in a connected component of $S_{2},$ then
\begin{equation*}
     BS\left(\left\{x\in X:S_{1}\subset A\left(\left(\frac{1}{n}\sum_{k=0}^{n-1}\phi_{j}(T^{k}x)\right)_{j\in I}\right)\right\},\varphi\right)
     =\inf_{x\in S_{1}}\sup_{\Xi \mu=x\atop\mu\in M(X,T)}\left\{\frac{h(T,\mu)}{\int\varphi d\mu}\right\}.
\end{equation*}

 \item
     If $S_{1}\neq\emptyset$ and $\overline{co}(S_{1})$ is not contained in a connected component of $S_{2},$ then
 \begin{equation*}
    \left\{x\in X:S_{1}\subset A\left(\left(\frac{1}{n}\sum_{k=0}^{n-1}\phi_{j}(T^{k}x)\right)_{j\in I}\right)\right\}=\emptyset.
 \end{equation*}
  \end{enumerate}
\end{cro}

\begin{cro}\label{cor5.10}
$(X,T,L_{n})$ as before. Let $(Y_{i},\Xi_{i})_{i}$ be (a possible
uncountable)
 family of deformations and assume that $Y_{i}$ is a normed vector space and that $\Xi_{i}:M(X)\rightarrow Y_{i}$ is affine and
 continuous. Define the vector spaces $\times_{i}Y_{i}$ and
 $[\times_{i}Y_{i}]^{\infty}$ by $$\times_{i}Y_{i}=\{(y_{i})_{i}|y_{i}\in Y_{i}\forall
 i\},$$$$[\times_{i}Y_{i}]^{\infty}=\{(y_{i})_{i}\in\times_{i}Y_{i}|\sup_{i}||y_{i}||<\infty\},$$
 and equip $[\times_{i}Y_{i}]^{\infty}$ with the norm
 $||(y_{i})_{i}||=\sup\limits_{i}||y_{i}||.$ Assume $\sup\limits_{\mu\in
 M(X),i}||\Xi_{i}\mu||<\infty$ and the map
 $$M(X)\rightarrow[\times_{i}Y_{i}]^{\infty}:\mu\mapsto(\Xi_{i}\mu)_{i}$$
 is continuous. Fix $S_{1},S_{2}\subset[\times_{i}Y_{i}]^{\infty}, \varphi\in C(X,\mathbb{R}^{+})$
 \begin{enumerate}
     \item
     If $S_{1}=\emptyset$ and $S_{2}$ is closed and convex, then
 \begin{equation*}
    BS(\{x\in X:S_{1}\subset A((\Xi_{j}L_{n}x)_{j\in I})\subset S_{2}\},\varphi)=
    \sup_{x\in S_{2}}\sup\limits_{\mu\in M(X,T)\atop(\Xi_{j}\mu)_{j\in I}=x}\left\{\frac{h(T,\mu)}{\int\varphi
    d\mu}\right\}.
\end{equation*}
      \item
     If $S_{1}\neq\emptyset$ and $\overline{co}(S_{1})$ is contained in a connected component of $S_{2},$ then
 \begin{equation*}
    BS(\{x\in X:S_{1}\subset A((\Xi_{j}L_{n}x)_{j\in I})\subset S_{2}\},\varphi)=\inf_{x\in S_{1}}\sup\limits_{\mu\in M(X,T)\atop(\Xi_{j}\mu)_{j\in I}=x}\left\{\frac{h(T,\mu)}{\int\varphi
    d\mu}\right\}.
\end{equation*}
        \item
     If $S_{1}\neq\emptyset$ and $\overline{co}(S_{1})$ is not contained in a connected component of $S_{2},$ then
 \begin{equation*}
   \{x\in X:S_{1}\subset A((\Xi_{j}L_{n}x)_{j\in I})\subset
   S_{2}\}=\emptyset.
\end{equation*}
 \end{enumerate}
\end{cro}

\subsection{The relative multifractal spectrum of ergodic averages }

 The relative multifractal spectrum of ergodic averages. Let $f,g\in
C(X,\mathbb{R})$ with $g(x)\neq0$ for all $x\in X$ and
$C\subseteq\mathbb{R}.$ Define $\Xi:M(X)\to \mathbb{R}$ by
$\Xi:\mu\mapsto\frac{\int f d\mu}{\int gd\mu}.$ Here remark that
$\Xi$ is continuous but not affine.

\begin{cro}\label{cor5.2}
 $(X,T,L_{n})$ as before. Let $f_1, g_1, \cdots, f_m, g_m$ be continuous functions $f_i, g_i:
X\to \mathbb{R}$ with $g_i(x)\neq0$ for all $x\in X, i=1,\cdots,m$
and $\int g_i d\mu\neq0,$ for all $\mu\in M(X,T),i=1,\cdots,m .$ If
$C\subseteq \mathbb{R}^m $ is closed and convex, $\psi\in
C(X,\mathbb{R}),$ then

\begin{equation*}\begin{split}
     &P\left(\left\{x\in X: A\left(\left(\frac{\sum_{k=0}^{n-1}f_{j}(T^{k}x)}{\sum_{k=0}^{n-1}g_{j}(T^{k}x)}\right)_{j\in \{1,2,\cdots,m\}}\right)\subseteq C\right\},\psi\right)
\\&=\sup\left\{h(T,\mu)+\int\psi d\mu:\mu\in M(X,T),\left(\frac{\int
f_i d\mu}{\int g_id\mu}\right)_{i\in\{1,\cdots,m\}}\in C\right\}.
\end{split}\end{equation*}
\end{cro}

\begin{proof}
Since the map $\Xi:\mu\mapsto \left(\frac{\int f_id\mu}{\int
g_id\mu}\right)_{i=1,\cdots,m}$ is continuous, we have

\begin{equation*}\begin{split}
&\left\{x\in X: A\left(\Xi L_n(x)\right)\subseteq C\right\}\\&
=\left\{x\in X: \Xi A\left(L_n(x)\right)\subseteq C\right\}\\&
=\left\{x\in X: A\left(L_n(x)\right)\subseteq \Xi^{-1}C\right\}\\&
\subseteq\left\{x\in X: A\left(L_n(x)\right)\cap
\Xi^{-1}C\neq\emptyset\right\}\\& =\left\{x\in X:
A\left(L_n(x)\right)\cap (\Xi^{-1}C\cap
M(X,T))\neq\emptyset\right\}.
\end{split}\end{equation*}
It follows from Proposition \ref{prop3.1} (i) that

\begin{equation*}\begin{split}
     &P\left(\left\{x\in X: A\left(\left(\frac{\sum_{k=0}^{n-1}f_{j}(T^{k}x)}{\sum_{k=0}^{n-1}g_{j}(T^{k}x)}\right)_{j\in \{1,2,\cdots,m\}}\right)\subseteq C\right\},\psi\right)
\\&\leq\sup\left\{h(T,\mu)+\int\psi d\mu:\mu\in M(X,T),\left(\frac{\int
f_i d\mu}{\int g_id\mu}\right)_{i\in\{1,\cdots,m\}}\in C\right\}.
\end{split}\end{equation*}
To the opposite inequality,  we prove the case $m=1$ as example. For
any $\alpha\in C,$

\begin{equation*}\begin{split}
&\left\{x\in X:A\left(\frac{1}{n}\sum\limits_{k=0}^{n-1}(f_1(T^kx)
-\alpha g_1(T^kx))\right)=0\right\}
\\&\subseteq\left\{x\in
X:A\left(\frac{\sum\limits_{k=0}^{n-1}f_1(T^kx)}{\sum\limits_{k=0}^{n-1}g_1(T^kx)}\right)\subseteq
C\right\}.
\end{split}\end{equation*}
Hence,

\begin{equation*}\begin{split}
&\sup\left\{h(T,\mu)+\int \psi d\mu:\frac{\int f_1 d\mu}{\int
g_1d\mu}=\alpha\in C,\mu\in
M(X,T)\right\}\\=&\sup\left\{h(T,\mu)+\int \psi d\mu:\int f_1-\alpha
g_1d\mu=0,\alpha\in C,\mu\in
M(X,T)\right\}\\=&\sup\limits_{\alpha\in C}P\left(\left\{x\in
X:A\left(\frac{1}{n}\sum\limits_{k=0}^{n-1}(f_1(T^kx) -\alpha
g_1(T^kx))\right)=0\right\},\psi\right)\\\leq&P\left(\left\{x\in
X:A\left(\frac{\sum\limits_{k=0}^{n-1}f_1(T^kx)}{\sum\limits_{k=0}^{n-1}g_1(T^kx)}\right)\subseteq
C\right\},\psi\right).
\end{split}\end{equation*}
Since the case $m>1$ is similar to $m=1,$ the proof is omitted.
\end{proof}

\begin{cro}\label{cor5.9}
 $(X,T,L_{n})$ as before. Let $f_1, g_1, \cdots, f_m, g_m$ be continuous functions $f_i, g_i:
X\to \mathbb{R}$ with $g_i(x)\neq0$ for all $x\in X, i=1,\cdots,m$
and $\int g_i d\mu\neq0,$ for all $\mu\in M(X,T),i=1,\cdots,m .$ If
$C\subseteq \mathbb{R}^m $ is closed and convex, $\varphi\in
C(X,\mathbb{R}^+),$ then

\begin{equation*}\begin{split}
     &BS\left(\left\{x\in X: A\left(\left(\frac{\sum_{k=0}^{n-1}f_{j}(T^{k}x)}{\sum_{k=0}^{n-1}g_{j}(T^{k}x)}\right)_{j\in \{1,2,\cdots,m\}}\right)\subseteq C\right\},\varphi\right)
\\&=\sup\left\{\frac{h(T,\mu)}{\int\varphi d\mu}:\mu\in M(X,T),\left(\frac{\int
f_i d\mu}{\int g_id\mu}\right)_{i=\{1,\cdots,m\}}\in C\right\}.
\end{split}\end{equation*}
\end{cro}

\subsection{symbolic space and iterated function systems}

Consider a subshift of finite type $\Sigma_A^+$ of the unilateral
full shift on $m$ symbols $I=\{1,2,\cdots,m\}$ with $m\geq2.$ Let
$\sigma $ be the shift map, and $A=(a_{ij})_{1\leq i,j\leq m}$ be
the transfer matrix of zeros and ones. In this section, we assume
that $A$ is an irreducible and aperiodic stochastic matrix, that is,
there is some power $m$ such that all the entries of $A^m$ are
strictly positive. This  assumption implies the specification
property.

For $x=(x_i)_{i\geq1}$ and $y=(y_i)_{i\geq1},$ set
$\nu(x,y)=\inf\{i\geq1: x_i\neq y_i\}.$ Let $\varphi$ be a strictly
positive continuous function on $\Sigma_A^+.$ Write $S_n
\varphi=\sum\limits_{i=0}^{n-1}\varphi\circ\sigma^i$ for each
$n\geq1.$ For $x\neq y\in \Sigma_A^+,$ define

\begin{equation*}
d_\varphi(x,y)=\left\{
  \begin{array}{ll}
    0, & x=y, \\
    1, &  x_1\neq y_1,\\
   \exp(-\min\limits_{\nu(x,z)\geq m}S_m \varphi(z)), & m=\nu(x,y).
  \end{array}
\right.
\end{equation*}
Remark that given $\Psi>1,$ we can choose $\varphi\equiv\ln \Psi,
d_\varphi$ is the metric in \cite{Ols2}.

\begin{prop}\label{prop6.1}
In $(\Sigma_A^+, d_\varphi),$ for any subset $Z\subset \Sigma_A^+,$
we get $dim_H(Z)= BS(Z,\varphi).$
\end{prop}
Let $\omega_{ij}$ be a Lipschitz contraction map on $\mathbb{R}^n$
for each nonzero $a_{ij}.$ There exists a unique vector
$E=(E_1,\cdots,E_m)$ of non-empty compact subsets of $\mathbb{R}^n$
satisfying $E_i=\bigcup\limits_{a_{ij}=1}\omega_{ij}(E_j).$ The
union $E=\bigcup\limits_{i=1}^mE_i$ is called a self-similar set for
recurrent iterated function system $\{\omega_{ij},(a_{ij})\}.$

Let $F$ be a compact subset of $E.$ Set $F_i=F\cap
E_i,i=1,\cdots,m,$ if vector $(F_1, F_2, \cdots, F_m)$ satisfying
$F_i\subseteq \bigcup\limits_{a_{ij}=1}\omega_{ij}(E_j),$ then the
set $F$ is called a sub-self-similar set for
$\{\omega_{ij},(a_{ij})\}.$

\noindent Assume that

\noindent (i) Each map $\omega_{ij}$ is a $C^{1+\gamma}$
diffeomorphism.

\noindent (ii) $D\omega_{ij}$ is always a similarity map, i.e.,
$|(D\omega_{ij})_x(\nu)|=s_{ij}(x)\cdot |\nu|$ for each
$x,\nu\in\mathbb{R}^n.$

\noindent (iii) $\{\omega_{ij},(a_{ij})\}$ satisfies the open set
condition \cite{Fal}.

Let $\pi:\Sigma_A^+\to E$ be given by

$\pi(x)=$ the only point in
$\bigcap\limits_{n\geq1}\omega_{x_ix_2}\omega_{x_2x_3}\cdots\omega_{x_{n-1}x_{n}}(E_{x_n}).$

The scale function of $E$ is the map $\psi:\Sigma_A^+\to \mathbb{R}$
given by $\psi(x)=\log s_{x_1x_2}(\pi\sigma x).$ Let
$\varphi(x)=-\psi(x),$ then $\varphi$ is a positive
H${\rm\ddot{o}}$lder continuous function.

\begin{prop}\label{prop6.2}
In $(\Sigma_A^+, d_\varphi),$ for any subset $Z\subset \Sigma_A^+,$
we get $dim_H(\pi Z)= dim_H( Z).$
\end{prop}
Combining Propositions \ref{prop6.1} and \ref{prop6.2}, Corollaries
\ref{cor5.1}, \ref{cor5.2},  \ref{cor5.3}, \ref{cor5.4},
\ref{cor5.5}, \ref{cor5.6}, \ref{cor5.7}, \ref{cor5.8},
\ref{cor5.9}, \ref{cor5.10} can hold about Hausdorff dimension in
iterated function system with open set condition. We take Corollary
\ref{cor5.9} as an example.

\begin{cro}\label{cro6.1}
 Let $f_1, g_1, \cdots, f_m, g_m$ be continuous functions $f_i, g_i:
\Sigma_A^+\to \mathbb{R}$ with $g_i(x)\neq0$ for all $x\in
\Sigma_A^+, i=1,\cdots,m$ and $\int g_i d\mu\neq0,$ for all $\mu\in
M(\Sigma_A^+,\sigma),i=1,\cdots,m .$ If $C\subseteq \mathbb{R}^m $
is closed and convex, then

\begin{equation*}\begin{split}
     &\dim_H\left(\pi\left\{x\in \Sigma_A^+: A\left(\left(\frac{\sum_{k=0}^{n-1}f_{j}(\sigma^{k}x)}{\sum_{k=0}^{n-1}g_{j}(\sigma^{k}x)}\right)_{j\in \{1,2,\cdots,m\}}\right)\subseteq C\right\}\right)
\\&=\sup\left\{\frac{h(T,\mu)}{-\int\log s_{x_1x_2}(\pi\sigma x) d\mu}:\mu\in M(\Sigma_A^+,\sigma),\left(\frac{\int
f_i d\mu}{\int g_id\mu}\right)_{i=\{1,\cdots,m\}}\in C\right\}.
\end{split}\end{equation*}
\end{cro}
Remark that our results are valid for sofic system (self-conformal
function system) induced by a subshift of finite type modelled by a
directed  and strongly connected multigraph. Similar to Corollary
\ref{cro6.1}, we give a positive answer to the conjecture in
\cite{Ols2} (see \cite{OlsWin2} for the view of Hausdorff dimension)
from the view of topological pressure.

%%%%%%%%%%%%%%%%%%%%%%%%%%%%%%%%%%%%%%%%%%%%%%%%%%%%%%%%%%%%%%%%%%%%%%%%%
%%%%%%%%%%%%%%%%%%%%%%%%%%%%%%%%%%%%%%%%%%%%%%%%%%%%%%%%%%%%%%%%%%%%%%%%%
%%%%%%%%%%%%%%%%%%%%%%%%%%%%%%%%%%%%%%%%%%%%%%%%%%%%%%%%%%%%%%%%%%%%%%%%%
%%%%%%%%%%%%%%%%%%%%%%%%%%%%%%%%%%%%%%%%%%%%%%%%%%%%%%%%%%%%%%%%%%%%%%%%%

\noindent {\bf Acknowledgements.} The authors would like to thank
Olsen and Winter for sharing their articles with us.  The work was
supported by the National Natural Science Foundation of China (grant
No. 11271191) and National Basic Research Program of China (grant
No. 2013CB834100) and the Foundation for Innovative program of
Jiangsu province (grant No. CXZZ12 0380).
%%%%%%%%%%%%%%%%%%%%%%%%%%%%%%%%%%%%%%%%%%%%%%%%


\footnotesize\noindent
\begin{thebibliography}{99}

\bibitem{BaeOlsSni} I. Baek, L. Olsen \& N. Snigireva, \emph{Divergence points of self-similar measures and packing dimension},
 Adv. Math. {\bf 214} (2007), 267-287.

\bibitem {BarSau} L. Barreira \& B. Saussol, \emph{Variational principles
and mixed multifractal spectra},  Trans. Amer. Math. Soc. {\bf 353}
(2001), 3919-3944.

\bibitem {BarSauSch} L. Barreira, B. Saussol \& J. Schmeling, \emph{Higher-dimensional
multifractal analysis},  J. Math. Pures Appl. {\bf 81} (2002),
67-91.

\bibitem {BarSch} L. Barreira \& J. Schmeling, \emph{Sets of "non-typical"
points have full topological entropy and full Hausdorff dimension},
Israel J. Math. {\bf 116} (2000), 29-70.

\bibitem{CheTasShu} E. Chen, Tassilo K\"{u}pper \& L. Shu,
\emph{Topological entropy for divergence points},  Ergod. Th. $\&$
Dynam. Sys. {\bf 25} (2005), 1173-1208.

\bibitem{CheXio} E. Chen \& J. Xiong, \emph{The pointwise dimension of
self-similar measures}, Chinese Sci. Bull. {\bf 44} (1999),
2136-2140.

\bibitem{Cli} V. Climenhaga, \emph{Topological pressure of simultaneous level
sets}, Nonlinearity {\bf 26} (2013), 241-268.

\bibitem{Fal} K. Falconer, \emph{Sub-self-similar sets}, Trans.
Amer. Math. Soc. {\bf 347} (1995), 3121-3129.


\bibitem {FanFen} A. Fan \& D. Feng, \emph{On the distribution of long-term time
averages on symbolic space},  J. Statist. Phys. {\bf 99} (2000),
813-856.

\bibitem {FanFenWu} A. Fan, D. Feng \& J. Wu, \emph{Recurrence, dimension and entropy},
 J. London Math. Soc. {\bf 64} (2001),
229-244.

\bibitem{FenHua} D. Feng \& W. Huang, \emph{Lyapunov spectrum of asymptotically sub-additive
potentials}, Commun. Math. Phys. {\bf 297} (2010), 1-43.

\bibitem {Oli1} E. Olivier, \emph{Analyse multifractale de fonctions continues}, C. R. Acad.
Sci. Paris Sr. I Math. {\bf 326} (1998), 1171-1174.

\bibitem {Oli2} E. Olivier, \emph{Structure multifractale d'une dynamique non expansive
d $\acute{e}$ finie sur unensemble de Cantor}, C. R. Acad. Sci.
Paris Sr. I Math. {\bf 331} (2000), 605-610.

\bibitem {Ols1} L. Olsen, \emph{Divergence points of deformed empirical
measures}, Mathematical Research Letters {\bf 9} (2002), 1-13.

\bibitem{Ols2} L. Olsen, \emph{Multifractal analysis of divergence points of
deformed measure theoretical Birkhoff averages}, J. Math. Pures
Appl. {\bf 82} (2003), 1591-1649.

\bibitem{Ols3} L. Olsen, \emph{Multifractal analysis of divergence points of
deformed measure theoretical Birkhoff averages. III}, Aequationes
Math. {\bf 71} (2006) 29-58.

\bibitem{Ols4} L. Olsen, \emph{Multifractal analysis of divergence points of
deformed measure theoretical Birkhoff averages. IV: Divergence
points and packing dimension},  Bull. Sci. Math. {\bf 132} (2008),
650-678.

\bibitem {OlsWin} L. Olsen \& S. Winter, \emph{Normal and non-normal points
of self-similar sets and divergence points of self-similar
measures}, Jour. London Math. Soc. {\bf 67} (2003), 103-122.


\bibitem {OlsWin2} L. Olsen \& S. Winter, \emph{Multifractal analysis of divergence points of deformed
measure theoretical Birkhoff averages. II: Non-linearity, divergence
points and Banach space valued spectra}, Bull. Sci. math.
 {\bf 131} (2007), 518-558.

\bibitem{Pes} Ya. Pesin, \emph{Dimension Theory in Dynamical Systems},   Contemporary Views and Applications.
University of Chicago Press, Chicago, IL, 1997.

\bibitem{PeiChe1} Y. Pei \& E. Chen, \emph{On the variational principle for the
topological pressure for certain non-compact sets},  Sci. China Ser
A {\bf 53(4)} (2010), 1117-1128.


\bibitem{PeiChe2} Y. Pei and E. Chen, \emph{Topological pressure for
divergence points}, preprint.

\bibitem{PesPit} Ya. Pesin and B.  Pitskel, \emph{Topological pressure and
the variational principle for noncompact sets},   Functional Anal.
Appl. {\bf 18} (1984), 307-318.

\bibitem{PfiSul1} C. Pfister \& W. Sullivan, \emph{Large deviations estimates for dynamical systems without the
specification property. Applications to the $\beta$-shifts},
Nonlinearity {\bf 18} (2005) 237-261.

\bibitem{PfiSul2} C. Pfister \& W. Sullivan, \emph{On the topological
entropy of saturated sets},  Ergodic Theory Dynam. Systems {\bf 27}
(2007) 929-956.

\bibitem{Rue} D. Ruelle, Historic behaviour in smooth dynamical
systems, in: H. Broer, B. Krauskopf, G. Vegter (Eds.), Global
Analysis of Dynamical Systems, Institute of Physics, London, 2001,
63-66.

\bibitem{Tak} F. Takens, \emph{Orbits with historic behavior, or non-existence of
averages}, Nonlinearity {\bf 21} (2008) 33-36.

\bibitem{TakVer} F. Takens \& E. Verbitskiy, \emph{On the variational principle for the topological entropy of certain non-compact sets},
Ergodic Theory Dynam. Systems {\bf 23(1)} (2003) 317-348.

\bibitem{Tho1} D. Thompson, \emph{The irregular set for maps with the specification property has full topological
pressure}, Dynamical Systems: An Internatioal Journal {\bf 25(1)}
(2010) 25-51.

\bibitem{Tho2} D. Thompson, \emph{A variational principle for topological pressure for certain non-compact
sets},  J. Lond. Math. Soc. {\bf 80(3)} (2009) 585-602.

\bibitem{Wal} P. Walters, An Introduction to Ergodic Theory, Springer, New York, 1982.

\bibitem{Win} S. Winter, \emph{ Convergence points and divergence points of
self-similar measures}, Dissertation, University of St Andrews,
2001.

\bibitem{Yam} Kenichiro Yamamoto, \emph{Topological pressure of the set of
generic points for $\mathbb{Z}^{d}$-actions}, Kyushu J. Math. {\bf
63} (2009), 191-208.

\bibitem{ZhoChe} X. Zhou \& E. Chen, \emph{Topological pressure of historic set for $\mathbb{Z}^{d}$-
actions}, J. Math. Anal. Appl. {\bf 389} (2012), 394-402.

\bibitem{ZhoCheChe} X. Zhou, E. Chen \& W. Cheng, \emph{Packing entropy and diveregence
points},  Dynamical Systems: An Internatioal Journal {\bf 27(3)}
(2012), 387-402.


\end{thebibliography}
\end{document}